\documentclass{article}
\usepackage[utf8]{inputenc}
\usepackage[top = 1in, bottom = 1in, left =1in, right = 1in]{geometry}
\usepackage{amsfonts,amscd,amssymb,amsmath,amsthm, mathrsfs}
\usepackage{graphicx,tikz}
\usepackage{xparse}

\usepackage{multicol}
\usepackage{floatrow}
\usepackage{nicefrac}
\usepackage{enumerate}
\usepackage{multirow}
\usepackage{amsthm}
\usepackage{thmtools}
\usepackage{fullpage}
\usepackage{url}
\usepackage{centernot}

\usepackage{marginnote}
\usepackage{verbatim}
\usepackage{mathtools}

\usepackage{wrapfig,subcaption}
\usepackage{hyperref}
\hypersetup{colorlinks=true,citecolor=blue}

\graphicspath{ {images/} }

\usepackage[
    backend=bibtex,
    style=alphabetic,
    sortlocale=deDE, 
    natbib=true,
    url=false, 
    doi=true,
    eprint=false,
    isbn=false,
    maxbibnames=10
]{biblatex}
\newcommand{\R}{\mathbb{R}}
\newcommand{\C}{\mathbb{C}}
\newcommand{\D}{\Diag}

\newcommand{\F}{\mathcal{F}}
\newcommand{\G}{\mathcal{G}}
\newcommand{\Q}{\mathbb{Q}}

\newcommand{\etr}{\operatorname{etr}}

\newcommand{\erfc}{\operatorname{erfc}}

\newcommand{\LP}{\operatorname{L}}

\newcommand{\sinc}{\operatorname{sinc}}

\declaretheorem{theorem}
\declaretheorem[sibling=theorem]{lemma}
\declaretheorem[sibling=theorem]{definition}
\declaretheorem[sibling=theorem]{proposition}
\declaretheorem[sibling=theorem]{corollary}

\declaretheorem[sibling=theorem]{conjecture}
\declaretheorem[sibling=theorem]{remark}

\newcommand{\TV}{\operatorname{TV}}

\renewcommand{\Re}{\operatorname{Re}}
\newcommand{\Image}{\operatorname{Im}}

\newcommand{\Id}{\operatorname{Id}}
\newcommand{\diag}{\operatorname{diag}}
\newcommand{\tr}{\operatorname{tr}}

\newcommand{\op}{\operatorname{op}}
\newcommand{\Fr}{\operatorname{F}}

\renewcommand{\Re}{\operatorname{Re}}
\renewcommand{\Im}{\operatorname{Im}}

\newcommand{\Exp}{\mathbb{E}}
\newcommand{\E}{\mathbb{E}}
\newcommand{\prob}{\mathbb{P}}

\renewcommand{\Pr}{\prob}
\DeclareDocumentCommand \one { o }
{%
\IfNoValueTF {#1}
{\mathbf{1}  }
{\mathbf{1}\{{#1}\} }%
}

\newcommand{\Sym}{\mathscr{M}}
\newcommand{\Diag}{\mathscr{D}}
\newcommand{\Offdiag}{\mathscr{E}}

\newcommand{\Binomial}{\operatorname{Binom}}

\newcommand{\lawequals}{\overset{\mathscr{L}}{=}}

\DeclareDocumentCommand{\Prto} {o} {
\IfNoValueTF {#1}
 {\overset{\Pr}{\longrightarrow}}
 { \xrightarrow[ #1 \to \infty]{\Pr }}
}
\DeclareDocumentCommand{\Asto} {o} {
\IfNoValueTF {#1}
 {\overset{\operatorname{a.s.}}{\longrightarrow}}
 {
 \xrightarrow[ #1 \to \infty]{\operatorname{a.s.} }
 }
}
\DeclareDocumentCommand{\Mgfto} {o} {
\IfNoValueTF {#1}
{\overset{\operatorname{mgf}}{\longrightarrow}}
{ \xrightarrow[ #1 \to \infty]{\operatorname{mgf} }}
}

\DeclareDocumentCommand{\Wkto} {o} {
\IfNoValueTF {#1}
 {\overset{(d)}{\longrightarrow}}
 { \xrightarrow[ #1 \to \infty]{(d) }}
}

\DeclareDocumentCommand \LPto { O{1} }
{\overset{\operatorname{\LP^{#1}}}{\longrightarrow}}

\title{Random geometric graphs and the spherical Wishart matrix}
\author{Elliot Paquette and Andrew Vander Werf}
\date{\today}

\begin{document}

\maketitle
\abstract{
  We consider the random geometric graph on $n$ vertices drawn uniformly from a $d$--dimensional sphere.  We focus on the sparse regime, when the expected degree is constant independent of $d$ and $n$.  We show that, when $d$ is larger than $n$ by logarithmic factors, this graph is comparable to the Erd\H{o}s--R\'enyi random graph of the same edge density in the \emph{inclusion divergence} between the graph laws.  This divergence functions in certain ways like a relaxation of the total variation distance, but is strong enough to distinguish Erd\H{o}s--R\'enyi graphs of different densities with a higher resolution than the total variation distance. To do the analysis, we derive some exact statistics of the \emph{spherical Wishart matrix}, the Gram matrix of $n$ independent uniformly random $d$--dimensional spherical vectors.  In particular we give expressions for the characteristic function of the spherical Wishart matrix which are well--approximated using steepest descent. 
}

\section{Introduction}
The \emph{random geometric graph} is defined by taking $n$ independent and identically distributed points in a metric space and connecting them if and only if their distance is sufficiently small.  Random geometric graphs have been the setting of extensive research, and they have a well--developed general theory (see \cite{penrose}).

We consider the random geometric graph $\G(n,p,d),$ in which $n$ points are sampled from the uniform measure (normalized Haar measure) on the unit sphere $\mathbb{S}^{d-1}$, with a distance threshold $t_{p,d}$ chosen in such a way that the probability that any two points connect is $p$.  Specifically, with $X$ uniformly distributed on the sphere and with $x$ any fixed point on the sphere,  $t_{p,d}\in[-1,1]$ is the unique value such that $\Pr[|x-X|\leq\sqrt{2-2t_{p,d}}]=p$, or equivalently such that $\Pr[\langle x,X\rangle\geq t_{p,d}]=p$.

Many authors in recent years have been interested in the problem of testing the hypothesis that a given graph $G$ has been chosen from the Erd{\H{o}}s--R\'enyi random graph $\G(n,p)$ or from $\G(n,p,d)$, and in particular on the effectiveness of such tests as $d\to\infty$.  The first paper to discuss this problem with $d\to\infty$
is \cite{Devroye}
which shows by an appeal to a central limit theorem that as $d\to\infty$ with $n$ and $p$ fixed, 
\begin{equation}\label{eq:TVtest}
\lim_{d\to\infty}\TV\big(\G(n,p),\G(n,p,d)\big)=0,
\end{equation}
where $\TV(\cdot,\cdot)$ is the total variation distance. In fact, the authors also show that this holds with $n\to\infty$ as long is $d=\omega(n^72^{n\choose 2})$.  It was also shown that, for $d$ which is only poly--logarithmic in $n,$ the clique number of $\G(n,p,d)$ matches what is seen in Erd{\H{o}}s--R\'enyi.  As the total variation distance can only contract when passing to statistics of the random graphs, \eqref{eq:TVtest} shows that in the setting $d\to \infty$ with $n,p$ held fixed, there is no statistic that can distinguish the two random graphs.  

The result \cite{Devroye} was greatly improved in \cite{bubeck} wherein the authors prove the following:
\begin{theorem}[{\cite{bubeck}}]
\begin{enumerate}[(a)]
\item 
Let $p\in(0,1)$ be fixed and suppose that $d = o(n^3)$. Then
$$\lim_{d\to\infty}\TV\big(\G(n,p),\G(n,p,d)\big)=1.$$
\item
If $d = \omega(n^3)$, we have
$$\lim_{d\to\infty}\sup_{p\in[0,1]}\TV\big(\G(n,p),\G(n,p,d)\big)=0.$$
\item
Let $c>0$ be fixed and suppose that $d = o(\log^3n)$. Then
$$\lim_{d\to\infty}\TV\big(\G(n,c/n),\G(n,c/n,d)\big)=1.$$
\end{enumerate}
\end{theorem}
Parts (a) and (b) together fully answer the problem of total variation distinguishability in the dense regime in which $p$ is fixed. Part (c) gives sufficient conditions on $d$ to ensure that the graphs can be distinguished in the sparse regime in which $p=\Theta(1/n)$, but it does not give sufficient conditions to ensure that the graphs are indistinguishable beyond the result in (b). They \cite{bubeck} conjecture, however, that $d=\omega(\log^3n)$ is the correct sufficient condition to ensure that the graphs are indistinguishable in the sparse regime.  This remains an open problem: the state of the art, due to \cite{Phase}, shows that when $d=\omega( n^{3/2}\log^{7/2}n)$, the graphs are indistinguishable in the sparse regime.

In this paper, we will consider a different comparison between $\G(n,p,d)$ and $\G(n,p)$.
Let $\mathbb{G}_{n}$ be the set of undirected graphs on $n$ vertices, and define the \emph{inclusion divergence}
\[
  \operatorname{I-Div}\big(\G(n,p,d) \vert \vert \G(n,p)\big)
  :=\min_{A \subset \mathbb{G}_n} 
  \biggl\{
  \max_{ G \in A} 
  \left|\frac{\Pr[\G(n,p,d)\supseteq G]}{\Pr[\G(n,p)\supseteq G]}-1\right|
  +\Pr[ \G(n,p) \in A^c]
  \biggr\}.
\]
This can be considered as a type of divergence measure between the laws, in that, should the inclusion divergence be $0,$ the laws are equal.  

Note that in the sparse regime $p=\Theta(1/n),$ $\G(n,p)$ is concentrated on a set of graphs whose individual inclusion probabilities are all $e^{-\Theta(n \log n)}.$  Hence for the inclusion divergence to be small, it must be that the law of $\G(n,p,d)$ matches $\G(n,p)$ with very high accuracy for a class of rare events.  In particular, there is no direct comparison possible between the inclusion divergence and total variation distance on spaces with arbitrarily small atoms. Nonetheless, it can be seen that two Erd\H{o}s--R\'enyi graphs are close in inclusion divergence only if they are close in total variation distance (see Theorem \ref{thm:idiver}).

For the inclusion divergence, we show:
\begin{theorem}\label{thm:idivergence}
  Suppose $p \sim c/n$ for some $c > 0$.  Then, if $d=\omega(n \log^5 n)$, 
  \[
    \operatorname{I-Div}\big(\G(n,p,d) \vert \vert \G(n,p)\big) \to 0.
  \]
\end{theorem}
\noindent We expect that the the same holds with total variation distance, but this is beyond our method.  Based on our approximation, it is natural to assume that the condition $d= \omega(n \log^5 n)$ is tight up to logarithmic factors, and that the statement in the theorem is already false when $d = \Theta(n \log^2 n)$ (see Conjecture \ref{conj:lower}).

\paragraph{Non--divergence formulation.} The proof of Theorem \ref{thm:idivergence} goes by a direct analysis of the inclusion probability of graphs.  Indeed, we identify a collection of graphs $A \subset \mathbb{G}_n$ for which we show:
\[
\lim_{d\to\infty}
\max_{G \in A} \left|\frac{\Pr[\G(n,p,d)\supseteq G]}{\Pr[\G(n,p)\supseteq G]}-1\right|=0.
\]
This class $A$ will be defined in terms of comparisons of certain graph statistics.

Throughout this paper we will use the following notation to reference certain statistics for the graph $G$ being considered:
\begin{multicols}{2}
\begin{enumerate}
  \item $\mu_G$ is the number of nonisolated vertices \emph{of $G$}. 
  \item $\sigma_G$ is the number of edges in $G$.
  \item $\delta_G$ is the maximum degree of $G$.
  \item $\tau_G$ is the number of triangles in $G$.
\end{enumerate}
\end{multicols}
\noindent It should be assumed that all of these statistics, as well as $p$, may vary with $d$ as $d\to\infty$.  Note that for the inclusion probability $\Pr[\G(n,p)\supseteq G]$, an isolated vertex may be removed from the graph $G$ without affecting the probability, and hence where convenient, we will assume that $G$ does not contain any isolated vertices.

Using these statistics we show:
\begin{theorem}\label{mainmain}
  Suppose that $d$ and $n$ tend to infinity and that $\log d=O(\log p^{-1})$.  Set $p_0 :=  1-\Phi(\sqrt{d}t_{p,d})$ where $\Phi$ is the standard normal distribution function.
Let $A=A_n$ be a class of graphs for which 
\[
  d=\omega(\sigma_G \delta_G^2 \log^2(\nicefrac{1}{p})+\mu_G\log \mu_G \delta_G^2 \log^2(\nicefrac{1}{p})+\tau_G^2 \sigma_G \log^3(\nicefrac{1}{p})),
  \quad 
  d \log d=\Omega(\mu_G \sigma_G^{-1} \delta_G^4 \log^2(\nicefrac{1}{p})). 
\]
Then
\[
\lim_{d\to\infty}
\max_{G \in A} \left|\frac{\Pr[\G(n,p,d)\supseteq G]}{\Pr[\G(n,p_0)\supseteq G]}-1\right|=0.
\]
\end{theorem}
\noindent This theorem implies Theorem \ref{thm:idivergence} on taking $A$ to be the graphs for which for sufficiently large $K(c),$ 
\[
  \mu_G \leq n,
  \quad \sigma_G \leq K(c) n,
  \quad \delta_G \leq \log n,
  \quad\text{and}\quad \tau_G \leq \log n.
\]
All of these properties are easily seen to hold on an event of probability tending to $1$ under $\G(n,p_0)$.  Using Lemma \ref{lem:pp0}, we may use $p_0$ or $p$ interchangeably for the $d$ considered in Theorem \ref{thm:idivergence}.  Note that for $d =o(n \log^2 n)$ this ceases to be true.

Theorem \ref{mainmain} includes some information about the low--dimensional regime as well. \begin{corollary}
  Suppose $\mu_G$ is bounded, and suppose that
$d=\omega(\log^3(\nicefrac{1}{p}))$. 
Then we have
$$\lim_{d\to\infty}\left|\frac{\Pr[\G(n,p,d)\supseteq G]}{\Pr[\G(n,p_0)\supseteq G]}-1\right|=0.$$ 
\end{corollary}
\noindent Recall the originally conjectured threshold in \cite{bubeck} of $d=\omega( \log^3(\nicefrac{1}{p}))$ for indistinguishability in the sparse regime. 

\subsection{Related Work}

\paragraph{Hypothesis testing.} There has been quite a bit of interest in the past decade or so in hypothesis testing for graphs--see \cite{Devroye}, \cite{bubeck}, \cite{Gao}, \cite{Debarghya}, \cite{BanerjeeM17}, \cite{DBLP:conf/colt/BreslerN18}, \cite{Ryabko}, \cite{information}, \cite{Phase}, \cite{chhor2020sharp}, to name a few. Typically, one treats $\G(n,p)$ as the null hypothesis and some other model as the alternative hypothesis. One popular alternative model which appears in several related works (\cite{BanerjeeM17}, \cite{Gao}, \cite{Debarghya}, \cite{chhor2020sharp}) is the stochastic block model and, more generally, IER graphs, which are like $\G(n,p)$ in that they have independent edges, but which differ from $\G(n,p)$ by not requiring that each edge be equally likely to appear. These kinds of graphs are popular for modeling community structure in graphs, particularly social networks. Testing between such graphs and $\G(n,p)$ therefore typically revolves around checking the graph sample(s) for unusual or unusually frequent structures that are better explained by the presence of community bias than by pure indifference. 

It has also been popular to consider whether or not a graph has geometric biases, that is, whether we can reliably tell that, instead of being sampled from $\G(n,p)$, it was sampled from a random geometric  graph with edge density $p$ but with vertices constrained either by the geometry of the space or by the way they are distributed over the space \cite{bubeck}, \cite{Devroye},
\cite{Debarghya}, \cite{information}, \cite{Phase}.
For telling such geometric graphs apart from $\G(n,p)$, one typically looks at some  statistic of the sampled graph, often one related to the number of triangles in a sample since geometric graphs tend to form triangles more frequently. This has necessitated the development of techniques for evaluating interesting statistics in all kinds of random geometric graph models \cite{Cliques}, \cite{bubeck}, 
\cite{Devroye}, \cite{information}, \cite{annulusgraphs}, \cite{GBM}.

\paragraph{Distinguishability.} The most recent progress in the direction of determining when $\G(n,p,d)$ in particular is indistinguishable from $\G(n,p)$ was made through a combination of geometric reasoning and information inequalities \cite{Phase}. However, before this, the largest leap forward in the convergence of random geometric graphs to $\G(n,p)$ was made by noticing the relationship that often exists between random geometric graphs and some variant of the Wishart matrix \cite{bubeck},  \cite{information}. By coupling this relationship with results such as \cite{Entropic}and \cite{Smooth} which prove various central limit theorems for Wishart matrices, the authors of \cite{bubeck} and  \cite{information} were able to compare random geometric graphs to $\G(n,p)$ by comparing the appropriate Wishart matrix to an appropriately scaled GOE matrix.  Recently, \cite{LiuRacz} extended this analysis to a class of smooth interpolations $\G(n,p,d,q)$ which allow for a smooth, tunable connectivity function that interpolates between $\G(n,p)$ and $\G(n,p,d)$.

In this paper, we too take note of the relationship between $\G(n,p,d)$ and a variant of the Wishart matrix. This variant, which we call the spherical Wishart matrix, has a more direct relationship to $\G(n,p,d)$ than the standard Wishart matrix, but the possible advantages of using this variant have so far been overshadowed by the amount of information already known about the standard Wishart matrix. 

Another, recent work \cite{BBH} considers comparisons between Wishart and GOE in a masked sense.  There, a subset of entries of a Wishart matrix are compared to jointly independent normals in total variation distance.  The support of that subset is called the \emph{mask}, and precise phase transitions for total variation distance are established in terms of the masking graph's properties.  For example, \cite[Theorem 2.5]{BBH} gives a sufficient condition for masked total variation distance, which shows that for $G$ satisfying similar conditions as in Theorem \ref{mainmain}, the masked total variation distance between Wishart and GOE tends to $0.$  While Theorem \ref{mainmain} also concerns the marginals of a (almost) Wishart matrix of a similar mask,  Theorem \ref{mainmain} also addresses rare events for this marginal, which are invisible in a total variation comparison of the graphs $\G(n,p)$ and $\G(n,p,d)$ restricted to $G$.  

\subsection{Layout}
The paper is organized as follows. In Section \ref{sec:idiv} we make some basic estimates about inclusion divergence.  In Section \ref{sec:p0} we go over some basic results about $\G(n,p,d)$ including bounds on $t_{p,d}$ and $p_0$ which will be used frequently.  In Section \ref{sec:wishart} we introduce the spherical Wishart matrix and its role in the bounds to come.  In Section \ref{sec:steepest} we introduce some steepest descent contours over which we compute the graph inclusion probability.  In Section \ref{sec:bounds} we use Fourier analysis as well as the results of the previous sections to bound $\Pr[\G(n,p,d)\supseteq G]$ in terms of $\Pr[\G(n,p_0)\supseteq G]$. 
In the Appendix, we review contour deformation insofar as we will need it for the paper.

\subsection{Acknowledgements} 
The first author is supported by an NSERC Discovery grant.  
The second author was partially supported by an NSF grant DMS-1547357.

\section{Inclusion divergence}
\label{sec:idiv}
We recall the inclusion divergence was defined, for random graphs $\mathcal{X}$ and $\mathcal{Y},$
\[
  \operatorname{I-Div}\big(\mathcal{X} \vert \vert \mathcal{Y} \big)
  =\min_{A \subset \mathbb{G}_n} 
  \biggl\{
  \max_{ G \in A} 
  \left|\frac{\Pr[\mathcal{X} \supseteq G]}{\Pr[\mathcal{Y} \supseteq G]}-1\right|
  +\Pr[ \mathcal{Y} \in A^c]
  \biggr\}.
\]
In this section, we make a simple comparison for how the inclusion divergence compares to the total variation distance.  In particular we show that the inclusion divergence distinguishes Erd\H{o}s--R\'enyi graphs with the greater granularity than the total variation distance throughout the sparse regime.

\begin{theorem}\label{thm:idiver}
  Suppose $p_n$ and $q_n$ are sequences in $[0,\tfrac 12]$ which tend to 0, and suppose that $n^2q_n\to \infty$. Then $\operatorname{I-Div}(\mathcal{G}(n,p_n) \vert \vert\mathcal{G}(n,q_n)) \to 0$ 
if and only if $|p_n-q_n|n^2 \to 0$. \end{theorem}
\begin{proof}
Let $\mathcal{X}\lawequals\G(n,p_n)$, $\mathcal{Y}\lawequals\G(n,q_n)$, and let $|G|$ denote the number of edges in a graph $G$. Writing out the inclusion divergence, we have 
$$\operatorname{I-Div}(\mathcal{X} \vert \vert\mathcal{Y})=\max_{ G \in A} 
  \left|\frac{\Pr[\mathcal{X}\supseteq G]}{\Pr[\mathcal{Y} \supseteq G]}-1\right|
  +\Pr[ \mathcal{Y} \in A^c]$$
for some $A\subset\mathbb{G}_n$.
Since, for any $G$, $$\left|\frac{\Pr[\mathcal{X} \supseteq G]}{\Pr[\mathcal{Y}\supseteq G]}-1\right|=\left|\frac{p_n^{|G|} }{q_n^{|G|}}-1\right|$$
which is monotone increasing in $|G|$, 
the inclusion divergence will take the form $$\Pr[ \mathcal{Y} \in A^c]+\left|\frac{p_n^{g} }{q_n^{g}}-1\right|$$
where $g$ is the maximum number of edges among graphs in $A$. Whatever this $g$ may be, we can minimize the contribution from $\Pr[ \mathcal{Y} \in A^c]$ by choosing $A$ to be the set of all graphs with at most $g$ edges. Therefore, letting $m={n\choose 2}$, 
$$\operatorname{I-Div}(\mathcal{X} \vert \vert\mathcal{Y})=\min_{g\leq m}\biggl\{\left|\frac{p_n^g }{q_n^g}-1\right|+\Pr[\Binomial(m,q_n)>g]\biggr\}.$$
If $\operatorname{I-Div}(\mathcal{X} \vert \vert \mathcal{Y}) < \epsilon$,
  for an $\epsilon <\tfrac 12,$
  we can conclude that $g\geq mq_n$. It follows that for every $\delta>0$ there is an $\epsilon_0>0$ sufficiently small that $\epsilon < \epsilon_0$ implies 
  \[
    |p_n/q_n-1|mq_n \leq |p_n/q_n-1|g <\delta.
  \]
  Therefore 
  \[
    \operatorname{I-Div}\big(\mathcal{X} \vert \vert \mathcal{Y} \big)
    \to 0
    \quad
    \implies
    \quad
    |p_n-q_n|n^{2} \to 0.
  \]
  Conversely,  if $|p_n-q_n| n^{2} \to 0$, we can take 
  $A$ to be the set of graphs with at most $mq_n+\omega\sqrt{mq_n(1-q_n)}$ edges for $\omega=\omega_n \to \infty$ growing slowly enough that
  \[
    \biggl|
    \frac{p_n^g}{q_n^g}
    -1\biggr|
    =
    \biggl|
    (1+o(1/(mq_n))^g
    -1\biggr|
    =o(1).
  \]
  By Chebyshev's inequality, $\Pr[\G(n,q_n)\in A^c]=o(1)$.

\end{proof}
\begin{remark}
This displays, as previously mentioned, that inclusion divergence distinguishes Erd{\"o}s--R{\'e}nyi graphs in this regime with greater resolution than total variation distance. It has been observed (see, for instance, the relevant example from the introduction of \cite{Janson}) that $\TV(\mathcal{G}(n,p_n) ,\mathcal{G}(n,q_n)) \to 0$ 
if and only if $|p_n-q_n|n^2=o(\sqrt{n^2q_n})$, a far weaker condition than requiring that $|p_n-q_n|n^2\to0$, given that we are assuming $n^2q_n\to\infty$.  
\end{remark}

\section{The edge inclusion probability}
\label{sec:p0}

We will start with our most basic computational tool, namely that we know explicitly the probability density function for the inner product of two uniform points on $\mathbb{S}^{d-1}$. \begin{lemma}\label{b&b}
Let $X,Y$ be independent uniform points on $\mathbb{S}^{d-1}$. For any $t\in\R$, 
$$\Pr[\langle X,Y\rangle\geq t]=\Pr[\langle X,Y\rangle\geq t|Y]=\Pr[X_1\geq t]=\int_t^1f_d(x)dx$$
where $$f_d(x):=\frac{\Gamma(d/2)}{\sqrt\pi\Gamma((d-1)/2)}(1-x^2)_+^{\frac{d-3}{2}}.$$ 
\end{lemma}
\begin{proof}
By rotational invariance, $Y\mapsto\Pr[\langle X,Y\rangle\geq t|Y]$ is a constant map. So in particular we have 
$$\Pr[\langle X,Y\rangle\geq t|Y]=\Pr[\langle X,e_1\rangle\geq t]=\Pr[X_1\geq t]$$
which gives the second inequality, and taking the expectation of this expression gives the first equality. 
The last equality just uses the known marginal density function of a coordinate of a uniform point on $\mathbb{S}^{d-1}$ with respect to Lebesgue measure.
\end{proof}
By noticing that $f_d(x)$ is very close to being the density of a mean zero, variance $d^{-1}$, normal random variable, we can achieve upper and lower bounds for $t_{p,d}$, which we recall is defined by the solution of 
\[
  p = \int_{t_{p,d}}^1 f_d(x)\,dx.
\]
We note that the density $(1-x^2)^{\tfrac{d-3}{2}}$ is log concave and moreover its density satisfies
\[
  \frac{d^2}{dx^2} \log ( (1-x^2)^{\tfrac{d-3}{2}} ) 
  = -(d-3)\frac{d}{dx}\frac{x}{1-x^2}
  = -(d-3)\frac{1+x^2}{(1-x^2)^2}
  \leq -(d-3).
\]
Thus by \cite[Theorem 11]{Caffarelli} or \cite[Theorem 9.2]{Wellner} there is an even transport map $T : \R \to \R$ so that $|T(x)| \leq |x|$ and so that $T(Z/\sqrt{d-3})$ has density $f_d$ for $Z$ a standard normal.  Hence, we have an exact domination for the tail functions:
\begin{equation}\label{eq:ub}
  \Pr( T(Z/\sqrt{d-3}) \geq t )
  \leq
  \Pr( Z \geq t\sqrt{d-3})
  \quad
  \text{for all }
  t > 0, Z \sim N(0,1).
\end{equation}
As a direct consequence, we have that
\begin{equation}\label{eq:p0}
  p 
  \leq
  \int_{t_{p,d}}^\infty e^{-(d-3)x^2/2}\sqrt{\frac{{d-3}}{{2\pi}}}dx
  \quad
  \text{and}
  \quad
  \sqrt{d-3}t_{p,d} \leq \Phi^{-1}(1-p),
\end{equation}
where $\Phi^{-1}$ is the quantile function of the normal.
Moreover, these bounds are close to sharp, as we show in the following lemma.
We note that other quantitative bounds are shown in
\cite[Lemma 2]{bubeck}, \cite[Lemma 5.1]{Phase} and \cite[Lemma 1]{Devroye}.
\begin{lemma}\label{lem:pp0}
Suppose $p\to0$ and $d=\omega(\log^2 p^{-1})$, and let $p_0:=1-\Phi(\sqrt{d}t_{p,d})$. Then 
$$t_{p,d}\sim\sqrt{\frac{2\log p^{-1}}{d}} \quad \text{and} \quad p=p_0(1-(1+o(1))dt_{p,d}^4/4).$$
\end{lemma}

\begin{remark}
We note that when $d = o(n \log^2n)$, it therefore follows that $|p_0-p|n^{2} \to \infty$, when $pn \to c > 0.$  This in particular implies that the difference between $p$ and $p_0$ is sufficiently large that $\mathcal{G}(n,p)$ and $\mathcal{G}(n,p_0)$ are no longer equivalent in the sense of inclusion divergence (see Theorem \ref{thm:idiver}).
\end{remark}
\noindent This leads us naturally to conjecture rather that:
\begin{conjecture}\label{conj:lower}
  Suppose that $p \sim c_0/n$ for some $c_0 > 0$ and that $d \sim c_1 n\log^2n$ for $c_1 > 0$.
  Then $\operatorname{I-Div}(\G(n,p,d) \vert\vert (\G(n,p_0))) \to 0$. 
\end{conjecture}
\noindent It is easily checked that 
\[
  \operatorname{I-Div}(\G(n,p,d) \vert\vert (\G(n,p_0))) \to 0
  \implies
  \operatorname{I-Div}(\G(n,p,d) \vert\vert (\G(n,p))) \centernot{\to} 0
\]
for $d,p$ as chosen in the conjecture.  

\begin{proof}
As we saw in (\ref{eq:p0}), $$t_{p,d}\leq\frac{\Phi^{-1}(1-p)}{\sqrt{d-3}}\sim\sqrt{\frac{2\log p^{-1}}{d}}.$$
Here and later we are using the fact that $\Phi^{-1}(1-p)^2\sim2\log p^{-1}$ for small $p$. 
For a working lower bound, note that we also have 
$$\frac{1}{2}-p\leq C_d\int_0^{t_{p,d}}e^{-(d-3)x^2/2}\sqrt{\frac{d-3}{2\pi}}\mathrm{d}x.$$
where $C_d \coloneqq \sqrt{\frac{2}{d-3}}\frac{\Gamma(d/2)}{\Gamma((d-1)/2)}\leq\frac{1}{1-2/d}$.
Thus 
$$p\geq\frac{1}{2}-C_d\left(\Phi(\sqrt{d-3}t_{p,d})-\frac{1}{2}\right)$$
and
$$t_{p,d}\geq\frac{\Phi^{-1}(1-p-d^{-1})}{\sqrt{d-3}}=\omega(d^{-1/2}).$$

To compare $p$ to $p_0$, let $C_d' \coloneqq \sqrt{\frac{2}{d}}\frac{\Gamma(d/2)}{\Gamma((d-1)/2)}=1+O(d^{-1})$.  
Then, letting $Z\lawequals\mathcal{N}(0,d^{-1})$, we have 
\begin{align*}
\int_{t}^\infty f_d(x)\mathrm{d}x&=C_d'\E\left[(1-Z^2)_+^{\frac{d-3}{2}}e^{\frac{d}{2}Z^2}\one\{Z\geq t\}\right] \\
&=C_d'\E\left[\left(1+\frac{3}{2}Z^2+\frac{15-2d}{8}Z^4+O(dZ^6)\right)\one\{t\leq Z<1\}\right].
\end{align*}
Writing 
$m_k:=\E[Z^k\one\{t\leq Z\leq 1\}]$, 
one can show by integration by parts that 
$$m_k=\frac{k-1}{d}m_{k-2}+\frac{t^{k-1}\phi(\sqrt{d}t)-\phi(\sqrt{d})}{\sqrt{d}}$$
where $\phi$ is the density of the standard normal, and $m_0=\Pr[t\leq Z\leq1]$.
Recall that for $x\to\infty$,  $$\frac{\phi(x)}{x}\sim1-\Phi(x)\leq e^{-x^2/2}.$$ Thus, assuming that $t=\omega(d^{-1/2})$, it is not hard to show from here that $m_{2k}=m_0(1+o(1))t^{2k}$. 
Thus 
$$\int_{t}^\infty f_d(x)\mathrm{d}x=m_0\left(1-(1+o(1))dt^4/4\right)= p_0\left(1-(1+o(1))dt^4/4\right).$$
Applying this to $t=t_{p,d}$ now gives the desired result. Using this, we can now provide a sharper lower bound for $t_{p,d}$ than we currently have. Specifically, we have 
$$p\geq\left(1-(1+o(1))\frac{\log^2p^{-1}}{d}\right)(1-\Phi(\sqrt{d}t_{p,d})).$$
Thus
$$t_{p,d}\geq(1-o(1))\sqrt{\frac{2\log p^{-1}}{d}}$$
provided that $d=\omega(\log^2p^{-1})$. 
\end{proof}

\section{The Spherical Wishart Matrix}
\label{sec:wishart}

Let $\Sym$ denote the ${{n+1} \choose 2}$--dimensional vector space of real symmetric $n\times n$ matrices. By identifying $\Sym$ with $\R^{{n+1} \choose 2}$, we can equip $\Sym$ with the Euclidean metric, the Borel sigma--algebra, Lebesgue measure, and a partial ordering $\succeq$ defined by writing $M\succeq N$ if and only if $M-N$ is positive semidefinite. We will also use $\C \otimes \Sym$ to denote the complex symmetric matrices.

We will be interested in two subspaces of $\Sym$, the $n$-dimensional subspace of real diagonal $n\times n$ matrices, denoted $\Diag$, and ${\Offdiag}:=\Diag^{\perp}$, the space of $n\times n$ real symmetric matrices with 0 along the diagonal.  Let $\diag:\Sym\to\D$ be the projection mapping onto $\D$. The inner product on $\Sym$ inherited from the inner product on $\R^{{n+1} \choose 2}$ can be expressed as 
$$\langle M,N\rangle:=\frac{1}{2}\tr(MN+\diag M\diag N).$$
For any $t \geq 0$ we will also let $\Diag_t$ denote those real diagonal matrices with all entries in $[-t,t]$.

\paragraph{Wishart matrices.}
One access point we have to information about the Gram matrix of $n$ i.i.d. uniformly random spherical vectors is its relationship with a more well--studied random matrix called the \emph{Wishart matrix} which, in its simplest form, is the Gram matrix of $n$ i.i.d. $d$--dimensional standard normal vectors. Denote this matrix by $W$. We will only discuss $W$ as much as it is needed; see Section 3.2 of \cite{muirhead2005aspects} for more details and proofs of the following formulas. 
Assuming $d \geq n$, the density of $W$ over $\Sym$ is given by 
\begin{equation}\label{Wdensity}
  f(Y):=\frac{\etr\left(-\frac{1}{2}Y\right)\det Y^{\frac{d-n-1}{2}}\one\{Y\succeq0\}}{2^{\frac{nd}{2}}\Gamma_n(d/2)},
\end{equation}
where $\Gamma_n(d/2):=\pi^{\frac{n(n-1)}{4}}\prod_{i=1}^n\Gamma\left(\frac{d-i+1}{2}\right)$ and $\etr :=\exp\tr$ is the exponential of the trace. The characteristic function of $W$ is given by 
\begin{equation}\label{eq:Wishart}
  \varphi_W(\Theta)=
  \E\exp\left(i\langle W,\Theta\rangle\right)=
  \det\big(\Id-i(\Theta+\diag\Theta)\big)^{-\frac{d}{2}}
\end{equation}
for any $\Theta\in\Sym$ (see \cite[Theorem 3.2.3]{muirhead2005aspects}).
Moreover, this extends to complex $\Theta$ as the following lemma shows.  
\begin{lemma}\label{lem:Wishart}
  For $\Theta = \Theta_1 + i \Theta_2$, where $\Theta_1,\Theta_2 \in \Sym$ satisfy $\Id + \Theta_2 + \diag \Theta_2 \succ 0$
  \[
    \Exp \exp\left( i \langle W,\Theta \rangle \right)
    = 
    \det\big(\Id-i(\Theta+\diag\Theta)\big)^{-\frac{d}{2}},
  \]
  and moreover the characteristic function is analytic in $\Theta$ in this domain.
\end{lemma}
\begin{proof}
  We reduce the problem to \cite[Theorem 3.2.3]{muirhead2005aspects}, which concerns $\Theta_2=0$ but allows $W$ to have a nontrivial covariance strucutre.  Observe that we are computing
  \[
    \Exp \exp\left( i \langle W,\Theta \rangle \right)
    =
    \etr\left( 
    \frac{i}{2} X^TX \Theta
    \right)
    =
    \Exp \exp\left( \frac{i}{2}\sum_{j=1}^d x_j^T \Theta x_j\right),
  \]
  where $X^TX = \sum_{j=1}^d x_j x_j^T$ is a representation of $W$ in terms of outer products of $d$ independent standard normals of dimension $n,$ and where we have used the cyclicity of the trace.
  Using independence, we therefore have
  \[
    \Exp \exp\left( i \langle W,\Theta \rangle \right)
    =
    \biggl(
    \Exp \exp\left( \frac{1}{2} x^T (i\Theta_1 - \Theta_2)x \right)
    \biggr)^{d}.
  \]
  Thus it suffices to consider the $d=1$ case, which we do going forward.
  Let $\Q$ be a new probability measure with Radon--Nikodym derivative 
  \[
    \frac{d\Q}{d\Pr} = 
    \exp\left( -\frac{1}{2} x^T \Theta_2 x \right)
    \det( \Id + \Theta_2 + \diag \Theta_2)^{1/2}.
  \]
  Under $\Q$, $x$ has inverse covariance matrix $\Sigma = (\Id + \Theta_2 + \diag \Theta_2)^{-1}$, and moreover,
  \[
    \Exp \exp\left( i \langle W,\Theta \rangle \right)
    =
    \Q( \exp\left( i \langle W,\Theta_1 \rangle \right) )
    \det( \Id + \Theta_2 + \diag \Theta_2)^{-1/2},
  \]
  where we have used $\Q(\cdot)$ to denote expectation with respect to $\Q.$
  Hence, we have reduced the problem to the real case with a nontrivial covariance.  This is done in \cite[Theorem 3.2.3]{muirhead2005aspects}, where it is shown
  \[
    \Q( \exp\left( i \langle W,\Theta_1 \rangle \right) )
    =
    \det( \Id - i(\Theta_1 + \diag \Theta_1)\Sigma)^{-1/2}.
  \]
  This completes the proof as
  \[
    \Exp \exp\left( i \langle W,\Theta \rangle \right)
    =
    \det( \Id - i(\Theta_1 + \diag \Theta_1)\Sigma)^{-1/2}
    \times
    \det( \Id + \Theta_2 + \diag \Theta_2)^{-1/2}
    =
    \det\big(\Id-i(\Theta+\diag\Theta)\big)^{-1/2}.
  \]
  Analyticity of moment generating functions can be concluded in general on the open set of $\Theta \in \C \otimes \Sym$ for which 
  \[
    \Exp |\exp\left( i \langle W,\Theta \rangle \right)|
    =
    \Exp \exp\left( \Re (i \langle W,\Theta \rangle) \right)
    < \infty.
  \]
\end{proof}

\paragraph{Spherical Wishart matrices.}
Recall that we could factor a Wishart matrix $W=X^TX$, which represents $W$ as a Gram matrix
\[
  W_{jk}= \langle x_j, x_k \rangle,
\]
where $X$ has columns $(x_1,x_2, \dots, x_n)$ given by independent standard normals in $\R^d.$  
Each such column can be factored as $x_j = \|x_j\|\tfrac{x_j}{\|x_j\|}$ so that $\|x_j\|$ is $\chi(d)$ distributed and $\tfrac{x_j}{\|x_j\|}$ is uniformly distributed on the sphere.  Moreover, all such lengths and spherical vectors become independent in this factorization.  As matrices, we can therefore decompose
\begin{equation}\label{law}
W = X^TX = D(\Id+V)D
\quad
\text{where}
\quad V_{jk} = \bigg\langle \frac{x_j}{\|x_j\|}, \frac{x_k}{\|x_k\|} \bigg\rangle
\quad
\text{for all $j\neq k$},
\end{equation}
and $D:=\sqrt{\diag W}$ is a diagonal matrix of i.i.d.\ $\chi(d)$ random variables independent of $V$.  We call $V$ the \emph{spherical Wishart matrix}.

When $d \geq n$ the spherical Wishart matrix admits a density.  Although it will not be used in this paper, we note the following:
\begin{lemma}
For $d \geq n$, the density of $V$ over ${\Offdiag}$ is given by 
$$Y\mapsto\frac{\Gamma(d/2)^n}{\Gamma_n(d/2)}\det(\Id+Y)^{\frac{d-n-1}{2}}\one\{\Id+Y\succ0\}.$$
\end{lemma}
\begin{proof}
  As the Wishart admits a density \eqref{Wdensity} on $\Sym,$ we may decompose it as the marginal density on $\Diag$ and the conditional density of $\Offdiag$ given the marginal density.  Moreover, by the independence of $D$ and $V$ in \eqref{law}, the density of $V$ is nothing but the conditional density of the Wishart on $\Sym$ given the diagonal is $\Id$.  Note the marginal density of the diagonal is the product density of $n$ independent $\chi^2(d)$ random variables, from which the expression follows. 
\end{proof} 

Instead of approaching the problem via its density, we shall approach the problem using the characteristic function, and we note that the representation we use allows for $d$ to be essentially arbitrary.  This representation is derived by partially inverse--Fourier transforming the Wishart characteristic function.

\begin{lemma}\label{chardef} 
For $d\geq3$, $\Theta\in{ \C \otimes \Offdiag}$ with $\Id + \Im \Theta \succ 0$ 
\begin{equation*}
\varphi_V(\Theta)= 
    \biggl(\frac{\Gamma(d/2)2^{d/2}e^{1/2}}{\pi}
    \biggr)^n
    \int_{\Diag}
    \etr(-i\Delta)
    \varphi_W(\Theta+\Delta)
    \mathrm{d}\Delta.
\end{equation*}
\end{lemma}
\begin{proof}
  Define the law of the random matrix $W_t$ to be that of $W$ conditioned on the event that all entries of $D$ in \eqref{law} are between $\sqrt{(1-t)_+}$ and $\sqrt{1+t}$.
  We can determine $\varphi_V$ by determining the characteristic function for $W_t-\Id$ and sending $t\to0$. 
  By the independence of $V$ from $D$, we have that $W_t \Wkto \Id+V$ as $t\to0$.  Moreover, as $W_t$ are bounded random variables, we have that for all $\Theta \in \C \otimes \Offdiag$ (which we recall have no entries on the diagonal),
  \[
    \Exp \exp( i\langle W_t, \Theta\rangle )
    \to
    \Exp \exp( i\langle V, \Theta\rangle )
    \eqqcolon\varphi_V(\Theta)
  \]
  as $t \to 0$.

  We suppose going forward that $\Theta\in{\C \otimes \Offdiag}$ has $\Id + \Im \Theta \succ 0$.
  Let $C_t$ be the probability that any particular diagonal entry of $W-\Id$ is in the interval $(-t,t)$.
  Since the diagonal entries of $W$ are i.i.d.\ $\chi^2(d)$, we have for $t \in [0,1]$ 
  \[
    C_t=\int_{1-t}^{1+t}
    \frac{x^{d/2-1}e^{-x/2}}{2^{d/2}\Gamma(d/2)}\mathrm{d}x
    =
    \frac{2^{1-d/2}}{\sqrt{e}\Gamma(d/2)}t+O(t^3)
  \]
  as $t \to 0.$

  By definition, we can express the Fourier--Laplace transform of $W_t$ at $\Theta$
  for any $t \in [0,1]$ as
  \begin{align}\label{tdub}
    \varphi_{W_t}(\Theta)&=
    \int_{\Diag}
    \int_{\Offdiag}
    \one_{\Diag_t}(D^2-\Id)
    \exp\left(i\langle\Theta, DVD\rangle\right) 
    \det(D^2)^{d/2-1}\etr(-D^2/2)\frac{\vartheta(\mathrm{d}V)\mathrm{d}(D^2)}{2^{nd/2}\Gamma(d/2)^n
    C_t^n},
  \end{align}
  where we have let $\vartheta$ represent the law of the spherical Wishart matrix.
  Recall that for Lebesgue--a.e.\ $D^2 \in \Diag$ 
  \[
    \one_{\Diag_t}(D^2-\Id)
    =
    \lim_{a\to\infty} 
    \int_{\Diag_t+\Id}
    (a/\pi)^n\sinc( a(D^2-X))\mathrm{d} X.
  \]
  Applying bounded convergence (with respect to integration against the law of $D^2$),
  \begin{equation}\label{tdub2}
    \varphi_{W_t}(\Theta)=
    \lim_{a\to\infty}
    \int\limits_{\Diag}
    \int\limits_{\Diag_t+\Id}
    \int\limits_{\Offdiag}
    \exp\left(i\langle\Theta, DVD\rangle\right) 
    \det(D^2)^{d/2-1}\etr(-D^2/2)
    (a/\pi)^n\sinc( a(D^2-X))
    \frac{\vartheta(\mathrm{d}V)\mathrm{d} X
    \mathrm{d}(D^2)}
    { 
    2^{nd/2}\Gamma(d/2)^n
    C_t^n
    }.
  \end{equation}
  We may now interchange these integrals freely, as they represent an expectation of a bounded random variable against a finite measure, and so we may bring the $\mathrm{d}X$ as the most interior.
  We then change the sinc integral using Fourier inversion to produce
  \[
    \begin{aligned}
    \int_{\Diag_t+\Id}
    (a/\pi)^n\sinc( a(D^2-X))\mathrm{d} X
    &=
    \prod_{j=1}^n
    \int\limits_{-a}^a
    \int_{-\infty}^{\infty}
    \frac{1}{2\pi}\exp( i \Delta_{jj}(D_{jj}^2-X_{jj}))
    \one_{X_{jj} \in [1-t,1+t]}
    \mathrm{d}X_{jj}
    \mathrm{d}\Delta_{jj} \\
    &=
    \prod_{j=1}^n
    \int\limits_{-a}^a
    \exp( i \Delta_{jj}(D_{jj}^2-1))
    \frac{\sin(t \Delta_{jj})}{ \pi \Delta_{jj}}
    \mathrm{d}\Delta_{jj} \\
    &=
    \int\limits_{\Diag_a}
    \exp( i \langle \Delta, D^2 -\Id \rangle)
    ( t/\pi)^n \sinc(t\Delta)
    \mathrm{d}\Delta,
    \end{aligned}
  \]
  from which we conclude (interchanging integrals once more)
  \begin{equation}\label{tdub3}
    \begin{aligned}
    \varphi_{W_t}(\Theta)&=
    \lim_{a\to\infty} \\
    &\int\limits_{\Diag_a}
    \int\limits_{\Offdiag}
    \int\limits_{\Diag} 
    ( t/\pi)^n \sinc(t\Delta)
    \exp\left(i\langle\Delta + \Theta, D(V+\Id)D - \Id\rangle\right) 
    \det(D^2)^{d/2-1}\etr(-D^2/2)
    \frac{ \vartheta(\mathrm{d}V)
      \mathrm{d}(D^2)
      \mathrm{d}\Delta
    }
    { 
      2^{nd/2}\Gamma(d/2)^n
      C_t^n
    }.
  \end{aligned}
  \end{equation}
  In summary, we have the representation
  \begin{equation*}
    \varphi_{W_t}(\Theta) 
    =
    \lim_{a\to\infty}
    C_t^{-n}
    \int\limits_{\Diag_a}
    ( t/\pi)^n \sinc(t\Delta)
    \etr(-i\Delta)
    \varphi_W(\Theta+\Delta)
    \mathrm{d}\Delta
  \end{equation*}
  The characteristic function $\varphi_W(\Theta)$ becomes integrable with respect to Lebesgue measure on $\Diag$ as soon as $d \geq 3.$  This can be seen as
  \[
    |\varphi_W(\Theta+\Delta)|
    =|\det(\Id -i\Theta -2i\Delta)|^{-d/2}
    =|\det(\Id-2i\Delta)|^{-d/2}|\det(\Id-i\Theta(\Id-2i\Delta)^{-1})|^{-d/2}.
  \]
  Thus outside of some compact set of $\Delta$ there is a $C(\Theta)>0$ so that
  \[
    |\varphi_W(\Theta+\Delta)|
    \leq C|\det(\Id -2i\Delta)|^{-d/2}.
  \]
  It follows that we may then take $a \to \infty$ to conclude
    \begin{equation}\label{eq:pF}
    \varphi_{W_t}(\Theta) 
    =
    C_t^{-n}
    \int_{\Diag}
    ( t/\pi)^n \sinc(t\Delta)
    \etr(-i\Delta)
    \varphi_W(\Theta+\Delta)
    \mathrm{d}\Delta.
  \end{equation}
  Taking the limit as $t\to 0,$ from dominated convergence, we have $\sinc(t\Delta) \to 0$ and so
  \[
    \varphi_{W_t}(\Theta) 
    \to
    \biggl(\frac{\Gamma(d/2)2^{d/2}e^{1/2}}{2\pi}
    \biggr)^n
    \int_{\Diag}
    \etr(-i\Delta)
    \varphi_W(\Theta+\Delta)
    \mathrm{d}\Delta.
  \]

\end{proof}

Since the characteristic function $\varphi_V$ is entire, we note that it is possible to deform the contour of integration in the previous representation to give an analytic continuation to all $\Theta$.

\begin{lemma}\label{mainchardef} 
  For $d\geq3$, $\Theta\in{\C\otimes\Offdiag}$ and $\lambda>0$ so that $\lambda \Id + \frac{1}{d}\Im \Theta \succ 0$ the characteristic function of $V$ is given by 
\begin{align}\label{generaldef}
\varphi_V(\Theta)=\left(\frac{\Gamma(d/2+1)}{2\pi i(d/2)^{\frac{d}{2}}}\right)^n\int_{\lambda \Id+i\Diag}\left(\frac{\etr Z}{\det Z}\right)^{\frac{d}{2}}\det\left(\Id-\frac{i}{d}\Theta Z^{-1}\right)^{-\frac{d}{2}}\mathrm{d}Z. 
\end{align}
\end{lemma}
\begin{proof} 
  It suffices to show the identity for $\Im \Theta \succeq 0,$ for having done so, we have given a representation which is analytic in this domain and agrees with $\varphi_V$, which is entire.  Thus, by the Identity Theorem, we can extend to the domain of analyticity of the representation, which is $\lambda\Id+\frac{1}{d}\Image\Theta\succ0$. 

  Using Lemma \ref{chardef}
\begin{align*}
\varphi_V(\Theta)
=\left(\frac{\Gamma(d/2)2^{d/2}e^{1/2}}{2\pi}\right)^n\int_{\Diag}\frac{\etr(-i\Delta)
}{\det\big(\Id-2i\Delta-i\Theta\big)^{\frac{d}{2}}}\mathrm{d}\Delta 
=\left(\frac{\Gamma(d/2)}{2\pi}\right)^n\int_{\Diag}\frac{\etr\big(\frac{1}{2}\Id-i\Delta\big)}{\det\big(\frac{1}{2}\Id-i\Delta-\frac{i}{2}\Theta\big)^{\frac{d}{2}}}\mathrm{d}\Delta. 
\end{align*} 
By making the change of variables $\frac{1}{2}\Id-i\Delta=\frac{d}{2}Z$ we have
\begin{align*}
\varphi_V(\Theta)=\left(-\frac{\Gamma(d/2+1)}{2\pi i\left(d/2\right)^{\frac{d}{2}}}\right)^n\int_{d^{-1}\Id-i\Diag}\left(\frac{\etr Z}{\det Z}\right)^{\frac{d}{2}}\det\left(\Id-\frac{i}{d}\Theta Z^{-1}\right)^{-\frac{d}{2}}\mathrm{d}Z,
\end{align*}
where $d^{-1}\Id-i\Diag$ is the set $\{Z\in\D\oplus i\D:\Re Z=d^{-1}\Id\}$ with each coordinate oriented so as to move from $i\infty$ to $-i\infty$. By reversing the orientation of each coordinate, we introduce a factor of $(-1)^n$ and have 
\begin{align*}
\varphi_V(\Theta)=\left(\frac{\Gamma(d/2+1)}{2\pi i\left(d/2\right)^{\frac{d}{2}}}\right)^n\int_{d^{-1}\Id+i\Diag}\left(\frac{\etr Z}{\det Z}\right)^{\frac{d}{2}}\det\left(\Id-\frac{i}{d}\Theta Z^{-1}\right)^{-\frac{d}{2}}\mathrm{d}Z.
\end{align*}
At this point, the only difference between this formula and the one in the statement of the lemma is the real part of the path over which we are integrating. We will utilize Lemma \ref{worldmover} to get us to the desired formula. In the context of Lemma \ref{worldmover}, the function we are trying to integrate is
$$g(Z):=\left(\frac{\etr Z}{\det Z}\right)^{\frac{d}{2}}\det\left(\Id-\frac{i}{d}\Theta Z^{-1}\right)^{-\frac{d}{2}}=\left(\etr Z\right)^{\frac{d}{2}}\det\left(Z-\frac{i}{d}\Theta\right)^{-\frac{d}{2}},$$
our initial path is $\alpha(x)=d^{-1}+ix$, and our target path is $\beta(x)=\lambda+ix$. By the Spectral Theorem, $i\Image Z-\frac{i}{d}\Theta$ is guaranteed to have purely imaginary eigenvalues and thus $Z-\frac{i}{d}\Theta$ is invertible as long as $\Re Z>0$ for all $j$, and as long as this matrix is invertible, $g$ is holomorphic. Therefore the image of the smooth homotopy $H(x,y):=d^{-1}+(\lambda-d^{-1})y+ix$ stays within $g$'s domain of holomorphicity. It's clear that $\frac{\partial H}{\partial y}$ is uniformly bounded. For $1\leq k\leq n$, we have $$\int_{\alpha^{n-k}\times\beta^k}|g(Z)|\mathrm{d}Z=e^{\frac{n-k}{2}+\frac{kd\lambda}{2}}\int_{\alpha^{n-k}\times\beta^k}\det\left((\Re Z)^2+\left(\Image Z-\frac{1}{d}\Theta\right)^2\right)^{-\frac{d}{4}}\mathrm{d}Z<\infty$$
as long as $d\geq3$. We can also see that, if each coordinate of $\Re Z$ is between $d^{-1}$ and $\lambda$, then $g$ tends to 0 as the imaginary part of any one of its coordinates tends to $\pm\infty$. So, by Lemma \ref{worldmover}, this path deformation can be done, and this completes the proof. 
\end{proof}
We shall need bounds on the modulus of the characteristic function for large $\Theta$.  This we can do when $d$ is much larger than $n.$
The next corollary bounds $|\varphi_V|$ for such $\Theta$ when $\lambda=1$.
\begin{corollary}\label{absbound}
For $d>2n$ and $\Theta\in\C \otimes {\Offdiag}$ with $d\Id + \Image\Theta\succ 0$,
we have
$$|\varphi_V(\Theta)|\leq\left(\frac{\Gamma(d/2+1)e^{\frac{d}{2}}}{2\sqrt{\pi} (d/2)^{\frac{d}{2}}}\right)^n\frac{\Gamma\left(\frac{d-2n}{4}\right)}{\Gamma(d/4)}\frac{\|\Id+d^{-1}\Image\Theta\|_{\op}^n}{\det(\Id+d^{-1}\Image\Theta)^{\frac{d}{2}}}\left(1+\frac{\left\|\Re\Theta\right\|_{\Fr}^2}{\|d\Id+\Image\Theta\|_{\op}^2}\right)^{\frac{n}{2}-\frac{d}{4}}.$$ 
\end{corollary}
\begin{proof}
We start with the formula we just derived with $\lambda=1$. Assuming $d \Id+\Image\Theta\succ0$, we have 
\begin{align*}
|\varphi_V(\Theta)|&=\left|\left(\frac{\Gamma(d/2+1)}{2\pi i(d/2)^{\frac{d}{2}}}\right)^n\int_{\Id+i\Diag}\left(\frac{\etr Z}{\det Z}\right)^{\frac{d}{2}}\det\left(\Id-\frac{i}{d}\Theta Z^{-1}\right)^{-\frac{d}{2}}\mathrm{d}Z\right| \\ 
&\leq\left(\frac{\Gamma(d/2+1)\exp\left(\frac{d}{2}\right)}{2\pi (d/2)^{\frac{d}{2}}}\right)^n\det\left(\Id+d^{-1}\Image\Theta\right)^{-\frac{d}{2}}\int_{\Diag}\left|\det\left(\Id+i\frac{Z-d^{-1}\Re\Theta}{\Id+d^{-1}\Image\Theta}\right)\right|^{-\frac{d}{2}}\mathrm{d}Z. 
\end{align*}
When we write the fraction of matrices $A/B$ for positive definite $B,$ we define this to be $B^{-1/2}AB^{-1/2}$.  So in the instance above, the following symmetric matrix appears
$$\frac{Z-d^{-1}\Re\Theta}{\Id+d^{-1}\Image\Theta}=(\Id+d^{-1}\Image\Theta)^{-\frac{1}{2}}(Z-d^{-1}\Re\Theta)(\Id+d^{-1}\Image\Theta)^{-\frac{1}{2}}.$$ 
Thus we can bound this determinant in terms of the Frobenius and operator norms by writing  
\begin{align*}
\left|\det\left(\Id+i\frac{Z-d^{-1}\Re\Theta}{\Id+d^{-1}\Image\Theta}\right)\right|^2&=\det\left(\Id+\left(\frac{Z-d^{-1}\Re\Theta}{\Id+d^{-1}\Image\Theta}\right)^2\right) \\
&\geq1+\left\|\frac{Z-d^{-1}\Re\Theta}{\Id+d^{-1}\Image\Theta}\right\|_{\Fr}^2\geq1+\frac{\|Z\|_{\Fr}^2+d^{-2}\|\Re\Theta\|_{\Fr}^2}{\|\Id+d^{-1}\Image\Theta\|_{\op}^2},
\end{align*}
where in the last step we have used that $Z$ and $\Re \Theta$ are orthogonal in the Frobenius inner product.
For brevity, we will start writing $A:=\Id+d^{-1}\Image\Theta$ and $a:=\|A\|_{\op}$. We now have 
\begin{align*}
|\varphi_V(\Theta)|&\leq\left(\frac{\Gamma(d/2+1)\exp\left(\frac{d}{2}\right)}{2\pi(d/2)^{\frac{d}{2}}}\right)^n\det A^{-\frac{d}{2}}\int_\D\left(1+\frac{d^{-2}\|\Re\Theta\|_{\Fr}^2+\|Z\|_{\Fr}^2}{a^2}\right)^{-\frac{d}{4}}\mathrm{d}Z \\
&=\left(\frac{\Gamma(d/2+1)\exp\left(\frac{d}{2}\right)}{2\pi (d/2)^{\frac{d}{2}}}\right)^n\det A^{-\frac{d}{2}}\left(1+a^{-2}d^{-2}\left\|\Re\Theta\right\|_{\Fr}^2\right)^{-\frac{d}{4}}\int_\D\left(1+\frac{\|Z\|_{\Fr}^2}{a^2+d^{-2}\|\Re\Theta\|_{\Fr}^2}\right)^{-\frac{d}{4}}\mathrm{d}Z \\ 
&=\left(\frac{\Gamma(d/2+1)\exp\left(\frac{d}{2}\right)}{2\pi (d/2)^{\frac{d}{2}}}\right)^na^n\det A^{-\frac{d}{2}}\left(1+a^{-2}d^{-2}\left\|\Re\Theta\right\|_{\Fr}^2\right)^{\frac{n}{2}-\frac{d}{4}}\int_\D\left(1+\|Z\|_{\Fr}^2\right)^{-\frac{d}{4}}\mathrm{d}Z \\ &=\left(\frac{\Gamma(d/2+1)\exp\left(\frac{d}{2}\right)}{2\sqrt{\pi} (d/2)^{\frac{d}{2}}}\right)^n\frac{\Gamma\left(\frac{d-2n}{4}\right)}{\Gamma(d/4)}a^n\det A^{-\frac{d}{2}}\left(1+a^{-2}d^{-2}\left\|\Re\Theta\right\|_{\Fr}^2\right)^{\frac{n}{2}-\frac{d}{4}}. 
\end{align*}
\end{proof}

\begin{figure}
    \centering
    \floatbox[{\capbeside\thisfloatsetup{capbesideposition={right,top},capbesidewidth=5cm}}]{figure}[\FBwidth]
	{\caption{Homotopy (circular arcs) from $1+ i\R$ to the contour $u \mapsto u \cot u + i u$ (defined on 
        $(-\pi,\pi)$). The curve is asymptotic to $-\infty \pm i \pi$ at $u = \pm \pi$.
	This homotopy avoids the (shaded) region where the function loses analyticity, which is contained within the unit disc. 
         }
\label{fig}}
	{\includegraphics[trim=0cm 0 3.5cm 0,clip,width=4cm]{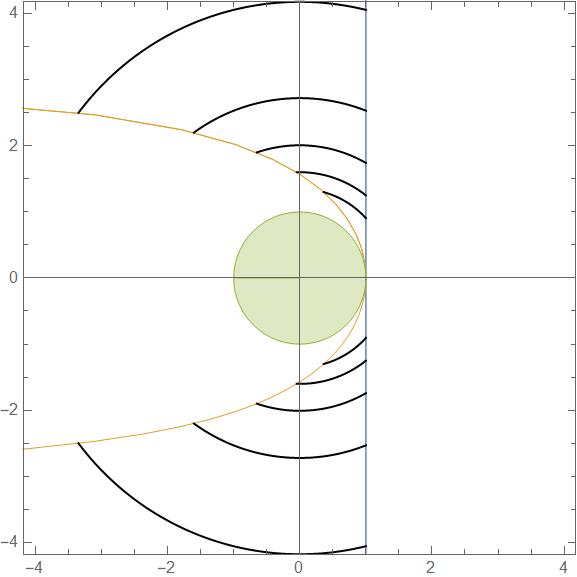}}
\end{figure}

\paragraph{Steepest descent.}
For $\Theta\in{\C \otimes \Offdiag}$ with $\|\Theta\|_{\op}<d$, we can deform $\lambda \Id+i\D$ in $(\ref{generaldef})$ with $\lambda=1$ to the contour $\beta$ parametrized by
$$\beta(U):=U\cot U+iU=\frac{e^{iU}}{\sinc U}$$
over $U\in\D_\pi$. To see this, note that, for $\|\Theta\|_{\op}<d$, all the singularities of 
$$Z\mapsto\left(\frac{\etr Z}{\det Z}\right)^{\frac{d}{2}}\det\left(\Id-\frac{i}{d}\Theta Z^{-1}\right)^{-\frac{d}{2}}$$ satisfy $\min_{j\in[n]}|Z_{jj}|=\|Z^{-1}\|^{-1}_{\op}<1$. Thus we can define a homotopy of the contour using circular arcs (see Figure \ref{fig}). 
The integral along these arcs is bounded by $\pi^ne^{nd}\sinc^{\frac{d}{2}-1}U$ which tends to 0 as any of the entries of $U$ tend to $\pm\pi$, assuming of course that $d\geq3$. 
Using this $\beta$, we now have  
\begin{align*}
\varphi_V(\Theta)=&\left(\frac{\Gamma(d/2+1)}{2\pi i(d/2)^{\frac{d}{2}}}\right)^n\int_\beta\left(\frac{\etr Z}{\det Z}\right)^{\frac{d}{2}}\det\left(\Id-\frac{i}{d}\Theta Z^{-1}\right)^{-\frac{d}{2}}\mathrm{d}Z \\ 
=&\left(\frac{\Gamma(d/2+1)}{2\pi i(d/2)^{\frac{d}{2}}}\right)^n\int_{\D_\pi}\left(\frac{\etr\beta(U)}{\det\beta(U)}\right)^{\frac{d}{2}}\det\left(\Id-\frac{i}{d}\Theta\beta(U)^{-1}\right)^{-\frac{d}{2}}\det\beta'(U)\mathrm{d}U. 
\end{align*}
Since $\Re\beta'(U)=\cot U-U\csc^2U$ is odd in each coordinate of $U$ and $\Image\beta'(U)=\Id$, it follows by setting $\Theta=0$ that 
\begin{align*}
\eta_d(U):=\left(\frac{\Gamma(d/2+1)}{2\pi (d/2)^{\frac{d}{2}}}\right)^n\left(\frac{\etr\beta(U)}{\det\beta(U)}\right)^{\frac{d}{2}}=\left(\frac{\Gamma(d/2+1)}{2\pi (d/2)^{\frac{d}{2}}}\right)^n\big(\etr(U\cot U)\det\sinc U\big)^{\frac{d}{2}}
\end{align*}
is a probability density on $\D_\pi$, 
and in particular $\eta_d(u):=\frac{\Gamma(d/2+1)}{2\pi(d/2)^{\frac{d}{2}}}\big(\exp(u\cot u)\sinc u\big)^{\frac{d}{2}}$ 
is a probability density on $(-\pi,\pi)$.
This proves the following lemma.
\begin{lemma}\label{expectationdef}
Let $U\in\D$ be a random diagonal matrix with density $\eta_d$. For $\Theta\in{\C \otimes \Offdiag}$ with $\|\Theta\|_{\op}<d$, 
$$\varphi_V(\Theta)=\E\left[\det\left(\Id-\frac{i}{d}\Theta \beta(U)^{-1}\right)^{-\frac{d}{2}}\det\beta'(U)\right]$$
where $\beta(x)=\frac{e^{ix}}{\sinc x}$.
\end{lemma}
\noindent The density $\eta_d$ is sufficiently well behaved that we can essentially disregard the $\beta'$ and treat $\beta(U)^{-1}\approx \Id + \mathcal{N}(0,cd^{-1}\Id)$ (see Lemma \ref{normalcomp}).

\section{Steepest Descent contour for Fourier inversion}
\label{sec:steepest}

Lemma \ref{expectationdef} is in a convenient form for estimating $\varphi_V(\Theta)$ when $\|\Theta\|_{\op} < d.$  We shall use this in particular to compare this characteristic function to an appropriately chosen Gaussian characteristic function.

Define $T:\Sym\to{\Offdiag}$ by writing $\big(T(\Theta)\big)_{jk}:=\one\{\Theta_{jk}\geq t_{p,d}\}$ so that $T(V)$ has the same law as the adjacency matrix of $\G(n,p,d)$. Now let $M=M(n,d)\in{\Offdiag}$ have i.i.d. $\mathcal{N}(0,d^{-1})$ upper triangular entries so that $T(M)$ has the same law as the adjacency matrix of $\G(n,p_0)$ where 
\[
  p_0
  \coloneqq 
  \Pr[ Z \geq \sqrt{d} t_{p,d}]
  =
  \frac{1}{2}\erfc\left(\sqrt{\frac{dt_{p,d}^2}{2}}\right),
  \quad
  \text{where $Z \lawequals \mathcal{N}(0,1)$}.
\]
The characteristic function of $M$ is
$$\varphi_M(\Theta):=\etr\left(-\frac{d}{4}\Theta^2\right).$$

For $j<k$, let $e_{jk}\in{\Offdiag}$ be 1 in its $(j,k)$--th and $(k,j)$--th entries and 0 elsewhere. For a graph $G$, let ${{\Offdiag}^G}\subseteq{\Offdiag}$ denote the subspace spanned by $\{e_{jk}:(j,k)\in G\}$. Then the characteristic function for the random vector $\big(V_{jk}:(j,k)\in G\big)$ is simply $\varphi_V$ restricted to ${\Offdiag}^G$ and the analogous statement also holds for $M$. Since we can write the event $\G(n,p,d)\supseteq G$ as having $V_{jk}\in[t_{p,d},\infty)$ for all $(j,k)\in G$ 
and $\G(n,p_0)\supseteq G$ as having $M_{jk}\in[t_{p,d},\infty)$ for all $(j,k)\in G$
, we have 
\begin{align}\label{gnpd}
\Pr\big[\mathcal{G}(n,p,d)\supseteq G\big]&=\lim_{a\to\infty}\int_{{\Offdiag}^G}\varphi_V(\Theta)\prod_{(j,k)\in G}\frac{e^{-it_{p,d}\Theta_{jk}}-e^{-ia\Theta_{jk}}}{2\pi i\Theta_{jk}}\mathrm{d}\Theta.
\end{align}
and 
\begin{align}\label{gnp}
\Pr\big[\mathcal{G}(n,p_0)\supseteq G\big]&=\lim_{a\to\infty}\int_{{\Offdiag}^G}\varphi_M(\Theta)\prod_{(j,k)\in G}\frac{e^{-it_{p,d}\Theta_{jk}}-e^{-ia\Theta_{jk}}}{2\pi i\Theta_{jk}}\mathrm{d}\Theta.
\end{align}

We will compare these probabilities.   To do so, we would like to change the Fourier inversion contour to one on which the Gaussian integral is non--negative.
\begin{proposition}{\label{odeprop}}
There exists a smooth curve $\gamma\in\{z\in\C:\Image z<0\}$ which is symmetric with respect to the imaginary axis for which $$\lim_{a\to\infty}\exp\left(-\frac{\gamma^2}{2d}\right)\frac{e^{-it_{p,d}\gamma}-e^{-ia\gamma}}{2\pi i\gamma}\gamma'=\exp\left(-\frac{\gamma^2}{2d}\right)\frac{e^{-it_{p,d}\gamma}}{2\pi i\gamma}\gamma'\geq0.$$
Moreover, this curve can be parametrized as $\gamma(x):=dt_{p,d}(x-iy(x))$ where $y'(x)\leq0$ for all $x\geq0$ and 
$$1\leq y\leq1+\frac{\arctan\left(x/y\right)}{dt_{p,d}^2x}\leq1+\frac{1}{dt_{p,d}^2}.$$ 
\end{proposition}
\begin{proof}
We want to find a $\gamma\in\{z\in\C:\Image z<0\}$ such that 
\begin{equation}\label{target}
\frac{\gamma'}{i\gamma}=\exp\left(i\Image\left(it_{p,d}\gamma+\frac{\gamma^2}{2d}\right)\right)\psi
\end{equation}
where $\psi$ is a positive function. We will parametrize $\gamma$ as above where we hope to have $y>0$. This makes $\gamma'(x)=dt_{p,d}(1-iy'(x))$. So (\ref{target}) becomes  
$$\frac{y-xy'}{x^2+y^2}-i\frac{x+yy'}{x^2+y^2}=\psi(x)\exp\left(idt_{p,d}^2(1-y)x\right).$$
Equating real and imaginary parts, this equality holds if and only if 
$$(y-xy')\sec\left(dt_{p,d}^2(1-y)x\right)=\psi(x)(x^2+y^2)=-(x+yy')\csc\left(dt_{p,d}^2(1-y)x\right).$$
Solving for $y'$ gives us the first--order nonlinear ordinary differential equation 
$$y'(x)=\frac{y(x)\tan(dt_{p,d}^2x(1-y))+x}{x\tan(dt_{p,d}^2x(1-y))-y(x)}=-\frac{\tan(dt_{p,d}^2x(1-y))+x/y}{1-\tan(dt_{p,d}^2x(1-y))x/y}$$
which, by using the formula for the tangent of a sum of angles, simplifies to 
\begin{equation}\label{ode}
y'(x)=-\tan\big(dt_{p,d}^2x(1-y)+\arctan(x/y)\big). 
\end{equation}
Peano's Existence Theorem states that for each point $(x_0,y_0)\in\R^2$ at which 
$$F(x,y):=-\tan\big(dt_{p,d}^2x(1-y)+\arctan(x/y)\big)$$
is continuous, we are guaranteed a local solution $y$ to (\ref{ode}) satisfying $y(x_0)=y_0$ which, for some $\varepsilon>0$, is defined and differentiable for $x\in(x_0-\varepsilon,x_0+\varepsilon)$ and which has a continuous extension to $[x_0-\varepsilon,x_0+\varepsilon]$. So, should we find that $F$ is still continuous in a neighborhood of the points $(x_0\pm\varepsilon,y(x_0\pm\varepsilon))$, then Peano's Theorem can be applied to find a local solution in a neighborhood of $(x_0\pm\varepsilon,y(x_0\pm\varepsilon))$ which differentiably extends our original local solution beyond $x_0\pm\varepsilon$ in either direction. This process of extending a local solution can then be repeated until the solution curve $(x,y(x))$ meets a point at which $F$ becomes discontinuous, and we would call this curve a maximal solution to (\ref{ode}). Of course, we would also like to make sure that our solution stays positive. So, in addition to halting this extension process once $F$ becomes discontinuous, we also will see fit to halt it if our solution becomes 0. 

Our plan will be to initiate the curve at some $(x_0,y_0)$ where  $x_0>0$ is arbitrarily large and $y_0$ depends on $x_0$. We will then show that this curve can be extended backwards to $-x_0$ and that these solutions defined on $[-x_0,x_0]$ converge as $x_0\to\infty$ to a well--defined curve on $(-\infty,\infty)$. In particular, for any $x_0>0$, let $y_0=y_0(x_0)$ be the unique solution to  
$$dt_{p,d}^2x_0(1-y_0)+\arctan(x_0/y_0)=0.$$
Let 
$$x_{\text{min}}:=\inf\big\{x\in(0,x_0):0<y(x)<\infty\text{ and }y'(x)=-\tan\big(dt_{p,d}^2x(1-y)+\arctan(x/y)\big)\text{ is finite}\big\}.$$
For our initial condition, take $y(x_0)=y_0$, though it is very important to note that we could set $y(x_0)$ to be anything in $(1,y_0]$ for what we are about to do. 

The first thing that we would like to prove is that $dt_{p,d}^2(1-y)x+\arctan\left(x/y\right)>0$ for  all $x\in(x_{\min},x_0)$. Suppose for contradiction that $\left\{x_{\min}<x<x_0:dt_{p,d}^2(1-y)x+\arctan\left(x/y\right)\leq0\right\}$ is nonempty, and set 
$$x_1:=\sup\left\{x_{\min}<x<x_0:dt_{p,d}^2(1-y)x+\arctan\left(x/y\right)\leq0\right\}.$$
By definition, $dt_{p,d}^2(1-y)x+\arctan\left(x/y\right)$ should have a non--negative derivative at $x_1$, and we should also have $dt_{p,d}^2(1-y(x_1))x_1+\arctan\left(x_1/y(x_1)\right)=0$ and  
$y'(x_1)=-\tan\left(dt_{p,d}^2(1-y(x_1))x_1+\arctan\left(x_1/y(x_1)\right)\right)=0$.  
Thus 
\begin{align*} 
x_1\frac{\mathrm{d}}{\mathrm{d}x}\left(dt_{p,d}^2(1-y)x+\arctan\left(x/y\right)\right)\Big\vert_{x=x_1}&=dt_{p,d}^2(1-y(x_1))x_1+\frac{y(x_1)x_1}{y(x_1)^2+x_1^2} \\ 
&=-\arctan\left(x_1/y(x_1)\right)+\frac{y(x_1)/x_1}{(y(x_1)/x_1)^2+1}<0 
\end{align*}
where the last inequality follows by noting that $\arctan(t^{-1})>\frac{t}{t^2+1}$ for all $t>0$. This contradicts $\left\{x_{\min}<x<x_0:dt_{p,d}^2(1-y)x+\arctan\left(x/y\right)\leq0\right\}$ being nonempty. We have therefore established that $dt_{p,d}^2(1-y)x+\arctan\left(x/y\right)>0$ for all $x\in(x_{\min},x_0)$ which also implies that $y'<0$ on this entire interval as well, thus making $y_0$ the global minimum of $y$ over $(x_{\min},x_0]$. So in total, 
\begin{align}\label{bounds}
    y_0<y< 1+\frac{\arctan(x/y)}{dt_{p,d}^2x}<1+\frac{1}{dt_{p,d}^2}
\end{align}
for $x\in(x_{\min},x_0)$. Additionally, it can be seen that we always have $y_0>1$. Therefore 
\begin{align}\label{morebounds}
    0<dt_{p,d}^2(1-y)x+\arctan\left(x/y\right)<\arctan(x)<\arctan(x_0)<\frac{\pi}{2}
\end{align} for all $x\in(x_{\min},x_0)$. This implies that 
$$x_{\text{min}}=\inf\big\{x\in(0,x_0):0<y(x)<\infty\text{ and }y'(x)=-\tan\big(dt_{p,d}^2x(1-y)+\arctan(x/y)\big)\text{ is finite}\big\}=0.$$ 
Not only does this show a differentiable solution to (\ref{ode}) with $y(x_0)=y_0(x_0)$ exists on $[0,x_0]$, but this solution is unique. Indeed, we can guarantee uniqueness over $[x_0-\varepsilon,x_0]$ for some $\varepsilon>0$ by the Picard--Lindelof Theorem. We claim that, as long as $x_0-\varepsilon>0$, we will have $1<y(x_0-\varepsilon)<y_0(x_0-\varepsilon)$. Thus we can repeat this process to uniquely extend this solution a little further to the left and continue this unique extension process until we reach 0. To see that $y(x_0-\varepsilon)<y_0(x_0-\varepsilon)$ as claimed, note that 
$$F(x_0-\varepsilon,y(x_0-\varepsilon))=y'(x_0-\varepsilon)<0=F(x_0-\varepsilon,y_0(x_0-\varepsilon))$$
and $F(x_0-\varepsilon,t)$ is increasing in the variable $t$ for the relevant range of $t$.  

In summary, the solution to (\ref{ode}) with initial condition $y(x_0)=y_0$ is uniquely defined, differentiable, satisfies (\ref{bounds}), and is monotone decreasing over $[0,x_0]$. 
Moreover, we can evenly extend this solution to $[-x_0,0]$ as follows: Let $\tilde{y}(x):=y(-x)$. Then, for $x\in[-x_0,0]$, we have
\begin{align*}
\tilde{y}'(x)&=-y'(-x) \\
&=\tan\big(-dt_{p,d}^2x(1-y(-x))-\arctan(x/y(-x))\big) \\ &=-\tan\big(dt_{p,d}^2x(1-\tilde{y})+\arctan(x/\tilde{y})\big).  
\end{align*} 
Since $x_0$ is a critical point for this solution, we can also extend this curve differentiably to $\R$ by setting $y(x)=y_0$ for $|x|>x_0$. Denote by $y(x;x_0)$ this curve which is a solution to (\ref{ode}) for $x\in[-x_0,x_0]$. \\

We still need to show that $\lim_{x_0\to\infty}y(x;x_0)$ converges to a well--defined solution to (\ref{ode}). First note that, as we showed earlier, for any $x_0'<x_0$ we have
$$y(x_0';x_0)< y_0(x_0')= y(x_0';x_0').$$
This implies that we must in fact have 
$$1<y(x;x_0)<y(x;x_0')$$
for all $x\leq x_0'<x_0$. Indeed, otherwise we would have two different solutions to (\ref{ode}) which go through $(x_0,y_0(x_0))$, contradicting the uniqueness of $y(x;x_0)$. By how we defined $y(x;x_0)$ for $x\geq x_0$ and the decreasing nature of the implicitly defined curve $y_0(x_0)$ and $y(x;x_0)$, it follows that 
$$1<y(x;x_0)<y(x;x_0')$$
for $x>x_0'$ as well. Thus, by (\ref{ode}) and (\ref{morebounds}), we have 
$$-x<\frac{\mathrm{d}}{\mathrm{d}x}y(x;x_0)<\frac{\mathrm{d}}{\mathrm{d}x}y(x;x_0')$$
for $x_0'<x_0$ and $x>0$. We therefore have measurable pointwise limits $y(x;\infty):=\lim_{x_0\to\infty}y(x;x_0)$ and  $\lim_{x_0\to\infty}\frac{\mathrm{d}}{\mathrm{d}x}y(x;x_0)$, which, by monotone convergence, satisfy  
$$y(x;\infty)-y(0;\infty)=\int_0^x\lim_{x_0\to\infty}\frac{\mathrm{d}}{\mathrm{d}s}y(s;x_0)\mathrm{d}s=\int_0^x-\tan\big(dt_{p,d}^2s(1-y(s;\infty))+\arctan(s/y(s;\infty))\big)\mathrm{d}s.$$
The proof is then completed by differentiating both sides with respect to $x$. 
\end{proof}

\begin{lemma}
For $\gamma$ as defined in the proof above, we have  $\left|\frac{\gamma'}{\gamma}\right|\leq1$. 
\end{lemma}
\begin{proof}
Since 
$0\leq dt_{p,d}^2(1-y)x+\arctan\left(x/y\right)\leq\arctan(x/y)$, the identity $\cos\arctan(x/y)=\frac{y}{\sqrt{x^2+y^2}}$ for $y>0$ gives $$\left|\frac{\gamma'}{\gamma}\right|^2=\frac{1+(y')^2}{x^2+y^2}=\frac{\sec^2\big(dt_{p,d}^2x(1-y)+\arctan(x/y)\big)}{x^2+y^2}\leq\frac{1}{y^2}\leq1.$$
\end{proof}

Let 
$$\gamma^G:=\{\Theta\in\C\otimes{\Offdiag}^G:\Theta_{jk}\in\gamma,\,(j,k)\in G\}.$$
Recall that we had parametrized $\gamma$ by writing $\gamma(x)=dt_{p,d}(x-iy(x))$ and letting $x$ run over all of $\R$. We can likewise parametrize $\gamma^G$ by writing 
$$\gamma^G_{jk}(X):=dt_{p,d}(X_{jk}-iY_{jk}(X))$$ 
where $Y_{jk}(X)=y(X_{jk})$ and $X$ runs over all of ${\Offdiag}^G$.

\begin{corollary}\label{density}
The map
$$x\mapsto\exp\left(-\frac{\gamma(x)^2}{2d}-it_{p,d}\gamma(x)\right)\frac{\gamma'(x)}{2\pi i p_0\gamma(x)}$$
is a probability density which is bounded above by $$\sqrt{\frac{dt_{p,d}^2}{2\pi }}\exp\left(o(1)-\frac{dt_{p,d}^2}{2}x^2\right).$$
\end{corollary}
\begin{proof}
The claim that this function is a probability density follows from the construction of $\gamma$. In particular, since $|\gamma'/\gamma|\leq1$ and $y\leq1+\frac{1}{dt_{p,d}^2}$,
\begin{align*}
\exp\left(-\frac{\gamma(x)^2}{2d}-it_{p,d}\gamma(x)\right)\frac{\gamma'(x)}{2\pi i p_0\gamma(x)}&=\exp\left(dt_{p,d}^2\left(\frac{y^2-x^2}{2}-y\right)\right)\frac{|\gamma'/\gamma|}{2\pi p_0} \\ 
&\leq e^{\frac{dt_{p,d}^2}{2}(1-y)^2}\frac{e^{-\frac{dt_{p,d}^2}{2}}}{2\pi p_0}\exp\left(-\frac{dt_{p,d}^2x^2}{2}\right) \\
&\leq e^{\frac{1}{2dt_{p,d}^2}}\frac{e^{-\frac{dt_{p,d}^2}{2}}}{2\pi p_0}\exp\left(-\frac{dt_{p,d}^2x^2}{2}\right).
\end{align*}
The rest follows by recalling that $p_0=  \frac{1}{2}\erfc\left(\sqrt{\frac{dt_{p,d}^2}{2}}\right)\sim\frac{e^{-\frac{dt_{p,d}^2}{2}}}{\sqrt{2\pi dt_{p,d}^2}}$. 
\end{proof}

\begin{lemma}
For $d\geq3$, 
\begin{equation}\label{mainequality}
    \Pr\big[\mathcal{G}(n,p,d)\supseteq G\big]=\Re\int_{{\Offdiag}^G}\varphi_V(\gamma^G)\prod_{(j,k)\in G}\frac{e^{-it_{p,d}\gamma^G_{jk}}}{2\pi i\gamma^G_{jk}}\frac{\partial\gamma^G_{jk}}{\partial X_{jk}}\mathrm{d}X
\end{equation}
and
$$\Pr\big[\mathcal{G}(n,p_0)\supseteq G\big]=\int_{{\Offdiag}^G}\varphi_M(\gamma^G)\prod_{(j,k)\in G}\frac{e^{-it_{p,d}\gamma^G_{jk}}}{2\pi i\gamma^G_{jk}}\frac{\partial\gamma^G_{jk}}{\partial X_{jk}}\mathrm{d}X$$
\end{lemma}

\begin{proof}
Recall we had previously concluded that 
\begin{align*}
\Pr\big[\mathcal{G}(n,p,d)\supseteq G\big]&=\lim_{a\to\infty}\int_{{\Offdiag}^G}\varphi_V(\Theta)\prod_{(j,k)\in G}\frac{e^{-it_{p,d}\Theta_{jk}}-e^{-ia\Theta_{jk}}}{2\pi i\Theta_{jk}}\mathrm{d}\Theta
\end{align*}
and 
\begin{align*}
\Pr\big[\mathcal{G}(n,p_0)\supseteq G\big]&=\lim_{a\to\infty}\int_{{\Offdiag}^G}\varphi_M(\Theta)\prod_{(j,k)\in G}\frac{e^{-it_{p,d}\Theta_{jk}}-e^{-ia\Theta_{jk}}}{2\pi i\Theta_{jk}}\mathrm{d}\Theta. \end{align*}
Moreover, we know that these integrands are entire and $\mathcal{L}^1$ over any path with bounded imaginary part. Thus we can deform each coordinate of ${\Offdiag}^G$ to any curve with bounded imaginary part 
and unbounded real part, such as the curve $\gamma$ defined in the proof of Proposition \ref{odeprop}.
Along a curve such as this one, we noticed that we have 
$$\lim_{a\to\infty}\frac{e^{-it_{p,d}\gamma}-e^{-ia\gamma}}{2\pi i\gamma}=\frac{e^{-it_{p,d}\gamma}}{2\pi i\gamma}$$  
due to the strict negativity of the imaginary part of $\gamma$. The result now follows by noting that the imaginary parts of these integrals must vanish since the left hand sides are probabilities. 
\end{proof}

\section{Bounding in terms of graph statistics}
\label{sec:bounds}

In this section, we prove Theorem \ref{mainmain}.
The error for the difference between 
$\Pr\big[\mathcal{G}(n,p,d)\supseteq G\big]$
and $\Pr\big[\mathcal{G}(n,p_0)\supseteq G\big]$ will be quantified in terms of four graph statistics of $G$.  
In this section we shall suppose without loss of generality $n=\mu_G,$ the number of vertices of the graph and that the graph has no isolated vertices.  We shall not explicitly suppose a relationship between $p$ and $n,$ but leave $p$ as a free parameter.
The other three graph parameters are the number of edges in $G$ denoted $\sigma_G$, the largest degree of $G$ denoted $\delta_G$, and the number of triangles in $G$ denoted $\tau_G$. Note that, since $n$ is the number of nonisolated vertices in $G$, we always have $n/2\leq \sigma_G \leq n\delta_G$. 

We will break the integral in \ref{mainequality} into two parts, one of which we will want to show is close to 1 and the other we will want to show is negligible in comparison. We will denote this negligible piece by
$$B(\kappa):=\left|\Re\int_{\{X\in{\Offdiag}^G:\|X\|_{\Fr}>\sqrt{2}\kappa\}}\varphi_V(\gamma^G)\prod_{(j,k)\in G}\frac{e^{-it_{p,d}\gamma^G_{jk}}}{2\pi i\gamma^G_{jk}}\frac{\partial\gamma^G_{jk}}{\partial X_{jk}}\mathrm{d}X\right|$$
where $\kappa$ is a cutoff over which we can optimize (but which will be chosen $o(d)$).


\subsection{Contribution from Large $\|X\|_{\Fr}$}

\begin{lemma}\label{Bbound}
Assume $\log d = O(dt_{p,d}^2)$, and suppose $G$ satisfies 
\[
  \sigma_G = o\left(t_{p,d}^{-2}\right),
  \quad\delta_G = o(t_{p,d}^{-1}), 
  \quad\frac{n \delta_G^4}{\sigma_G}= O\left(\frac{\log d}{dt_{p,d}^4}\right),
  \quad\text{and}
  \quad\frac{\tau_G}{\sigma_G}= O\left(\frac{\log d}{dt_{p,d}^3}\right).
\]
Then, for a sufficiently large constant $C$ and $\kappa=\sqrt{C\sigma_G}$, we have $B(\kappa)=O(p_0^{\sigma_G}/d)$. 
\end{lemma}
\begin{proof}
By Proposition \ref{odeprop} and our assumptions on $G$, we see that $\|\Image\gamma^G\|_{\op}=O(d\delta_Gt_{p,d}) =  o(d)$ and thus we can apply the results of Corollary \ref{absbound}.
So we have
\begin{align*}
&\left|\varphi_V(\gamma^G)\prod_{(j,k)\in G}\frac{e^{-it_{p,d}\gamma^G_{jk}}}{2\pi i\gamma^G_{jk}}\frac{\partial\gamma^G_{jk}}{\partial X_{jk}}\right| \\  \leq&\left(\frac{\Gamma(d/2+1)e^{\frac{d}{2}}}{2\sqrt{\pi} (d/2)^{\frac{d}{2}}}\right)^n\frac{\Gamma\left(\frac{d-2n}{4}\right)}{\Gamma(d/4)}\frac{\|I-t_{p,d}Y\|_{\op}^n}{\det(I-t_{p,d}Y)^{\frac{d}{2}}}\left(1+\frac{t_{p,d}^2\left\|X\right\|_{\Fr}^2}{\|I-t_{p,d}Y\|_{\op}^2}\right)^{\frac{n}{2}-\frac{d}{4}}\prod_{(j,k)\in G}\frac{e^{-dt_{p,d}^2Y_{jk}}}{2\pi |\gamma^G_{jk}|}\left|\frac{\partial\gamma^G_{jk}}{\partial X_{jk}}\right|. 
\end{align*}
We bound this as follows. By Stirling's formula, 
$$\left(\frac{\Gamma(d/2+1)e^{\frac{d}{2}}}{2\sqrt{\pi} (d/2)^{\frac{d}{2}}}\right)^n\frac{\Gamma\left(\frac{d-2n}{4}\right)}{\Gamma(d/4)}\leq\exp\left(\frac{n^2}{2(d-2n)}\right)=e^{O(\sigma_G\log d)}.$$
Since the entries of $Y$ are all bounded above by $1+\frac{1}{dt_{p,d}^2}$, we have that $$\|I-t_{p,d}Y\|_{\op}^n\leq(1+2t_{p,d}\delta_G)^n\leq e^{nt_{p,d}\delta_G}= e^{O(\sigma_G\log d)}$$
and
\begin{align*}
\log\det(I-t_{p,d}Y)^{-\frac{d}{2}}&\leq\tr(dt_{p,d}^2Y^2/4)+\tr(dt_{p,d}^3Y^3/6)+100ndt_{p,d}^4\|Y\|_{\op}^4 \\
&\leq\big(1+1/(dt_{p,d}^2)\big)^2\sigma_G dt_{p,d}^2/2+\big(1+1/(dt_{p,d}^2)\big)^3dt_{p,d}^3\tau_G+200ndt_{p,d}^4\delta_G^4 \\ 
&=\sigma_Gdt_{p,d}^2/2+O(\sigma_G\log d).
\end{align*}
As defined in Proposition \ref{ode}, the curve $\gamma(x)=dt_{p,d}(x-iy)$ satisfies 
$y\geq1$ and $\left|\frac{\gamma'}{\gamma}\right|\leq1$.
So 
$$\det(I-t_{p,d}Y)^{-\frac{d}{2}}\prod_{(j,k)\in G}\frac{e^{-dt_{p,d}^2Y_{jk}}}{2\pi |\gamma^G_{jk}|}\left|\frac{\partial\gamma^G_{jk}}{\partial X_{jk}}\right|=\left(\frac{e^{-dt_{p,d}^2/2}}{\pi}\right)^{\sigma_G}e^{O({\sigma_G}\log d)}.$$
Putting all of this together gives us a bound of 
$$
\left|\varphi_V(\gamma^G)\prod_{(j,k)\in G}\frac{e^{-it_{p,d}\gamma^G_{jk}}}{2\pi i\gamma^G_{jk}}\frac{\partial\gamma^G_{jk}}{\partial X_{jk}}\right|=
\left(\frac{e^{-dt_{p,d}^2/2}}{\pi}\right)^{\sigma_G}
e^{O(\sigma_G\log d)}
\left(1+\frac{t_{p,d}^2\left\|X\right\|_{\Fr}^2}{(1+2t_{p,d}\delta_G)^2}\right)^{\frac{n}{2}-\frac{d}{4}}.$$
We can write 
$$\int_{\{X\in{\Offdiag}^G:\|X\|_{\Fr}>\sqrt{2}\kappa\}}\left(1+\frac{t_{p,d}^2\left\|X\right\|_{\Fr}^2}{(1+2t_{p,d}\delta_G)^2}\right)^{\frac{n}{2}-\frac{d}{4}}\mathrm{d}X=\int_A(1+\alpha^2\|x\|_2^2)^{-q}\mathrm{d}x$$
where $A:=\{x\in\R^{\sigma_G}:\|x\|_2>\kappa\}$, and we set $q=\frac{d-2n}{4}$ and $\alpha=\frac{\sqrt{2}t_{p,d}}{1+2t_{p,d}\delta_G}$. The presence of the $\sqrt{2}$ in the expression for $\alpha$ is to account for the symmetry of $X$. Therefore this integral can be bounded as 
\begin{align*}
\int_A(1+\alpha^2\|x\|_2^2)^{-q}\mathrm{d}x&=\frac{2\pi^{{\sigma_G}/2}\alpha^{-{\sigma_G}}}{\Gamma({\sigma_G}/2)}\int_{\alpha\kappa}^\infty\frac{r^{{\sigma_G}-1}}{(1+r^2)^q}\mathrm{d}r \\
&\leq\frac{2\pi^{{\sigma_G}/2}\alpha^{-{\sigma_G}}}{\Gamma({\sigma_G}/2)}\int_{\alpha\kappa}^\infty\frac{r}{(1+r^2)^{q-{\sigma_G}/2+1}}\mathrm{d}r \\
&=\frac{\pi^{{\sigma_G}/2}\alpha^{-{\sigma_G}}}{(q-{\sigma_G}/2)\Gamma({\sigma_G}/2)}(1+(\alpha\kappa)^2)^{{\sigma_G}/2-q}.
\end{align*}
So altogether we have 
\begin{align*}
B(\kappa)&\leq\int_{\{X\in{\Offdiag}^G:\|X\|_{\Fr}>\sqrt{2}\kappa\}}\left|\varphi_V(\gamma^G)\prod_{(j,k)\in G}\frac{e^{-it_{p,d}\gamma^G_{jk}}}{2\pi i\gamma^G_{jk}}\frac{\partial\gamma^G_{jk}}{\partial X_{jk}}\right| \\
&\leq e^{O({\sigma_G}\log d)}\left(\frac{e^{-dt_{p,d}^2/2}}{\sqrt{\pi}t_{p,d}}\right)^{\sigma_G}\frac{(1+t_{p,d}^2\kappa^2)^{{\sigma_G}/2+n/2-d/4}}{(d-2n-2{\sigma_G})\Gamma({\sigma_G}/2)} \\ &\leq e^{O({\sigma_G}\log d)}\left(\frac{e^{-dt_{p,d}^2/2}}{\sqrt{\pi dt_{p,d}^2/2}}\right)^{\sigma_G}\frac{(1+t_{p,d}^2\kappa^2)^{-d/5}}{(d-2n-2{\sigma_G})\Gamma({\sigma_G}/2)} \\
&\leq p_0^{\sigma_G}\frac{e^{O({\sigma_G}\log d)}(1+t_{p,d}^2\kappa^2)^{-d/5}}{(d-2n-2{\sigma_G})\Gamma({\sigma_G}/2)}. 
\end{align*}
Here we use the fact that $\sigma_G=o(t_{p,d}^{-2})$  and $\kappa^2=C\sigma_G$ to write $$(1+\kappa^2t_{p,d}^2)^{-d/5}=e^{C(1+o(1))(dt_{p,d}^2\sigma_G/5)}.$$ 
Since $dt_{p,d}^2=\Omega\left(\log d\right)$, we can choose $C$ large enough that $B(\kappa)=O(p_0^{\sigma_G}/d)$. 

\end{proof}
\subsection{Contribution from Small $\|X\|_{\Fr}$}
Let 
$$K:=\{X\in{\Offdiag}^G:\|X\|_{\Fr}\leq\sqrt2 \kappa\}$$ where $\kappa=\sqrt{C{\sigma_G}}$ is as in the previous lemma. Then
\begin{align*}
&\Re\int_K\varphi_V(\gamma^G)\prod_{(j,k)\in G}\frac{e^{-it_{p,d}\gamma^G_{jk}}}{2\pi i\gamma^G_{jk}}\frac{\partial\gamma^G_{jk}}{\partial X_{jk}}\mathrm{d}X \\ 
=&\Re\int_K\left(1-1+\frac{\varphi_V(\gamma^G)}{\varphi_M(\gamma^G)}\right)\varphi_M(\gamma^G)\prod_{(j,k)\in G}\frac{e^{-it_{p,d}\gamma^G_{jk}}}{2\pi i\gamma^G_{jk}}\frac{\partial\gamma^G_{jk}}{\partial X_{jk}}\mathrm{d}X \\  
\leq&p_0^{\sigma_G}+p_0^{\sigma_G}\int_K\left|\frac{\varphi_V(\gamma^G)}{\varphi_M(\gamma^G)}-1\right|\varphi_M(\gamma^G)\prod_{(j,k)\in G}\frac{e^{-it_{p,d}\gamma^G_{jk}}}{2\pi ip_0\gamma^G_{jk}}\frac{\partial\gamma^G_{jk}}{\partial X_{jk}}\mathrm{d}X. 
\end{align*} 
By the previous lemma, we also have 
\begin{align*}
&\Re\int_K\varphi_V(\gamma^G)\prod_{(j,k)\in G}\frac{e^{-it_{p,d}\gamma^G_{jk}}}{2\pi i\gamma^G_{jk}}\frac{\partial\gamma^G_{jk}}{\partial X_{jk}}\mathrm{d}X \\ 
=&\Re\int_K\left(1-1+\frac{\varphi_V(\gamma^G)}{\varphi_M(\gamma^G)}\right)\varphi_M(\gamma^G)\prod_{(j,k)\in G}\frac{e^{-it_{p,d}\gamma^G_{jk}}}{2\pi i\gamma^G_{jk}}\frac{\partial\gamma^G_{jk}}{\partial X_{jk}}\mathrm{d}X \\  \geq&(1-o(d^{-1}))p_0^{\sigma_G}-p_0^{\sigma_G}\int_K\left|\frac{\varphi_V(\gamma^G)}{\varphi_M(\gamma^G)}-1\right|\varphi_M(\gamma^G)\prod_{(j,k)\in G}\frac{e^{-it_{p,d}\gamma^G_{jk}}}{2\pi ip_0\gamma^G_{jk}}\frac{\partial\gamma^G_{jk}}{\partial X_{jk}}\mathrm{d}X. 
\end{align*} 
Thus 
$$\big|\Pr\big[\mathcal{G}(n,p,d)\supseteq G\big]-p_0^{\sigma_G}\big|\leq p_0^{\sigma_G}\left(o(d^{-1})+\int_K\left|\frac{\varphi_V(\gamma^G)}{\varphi_M(\gamma^G)}-1\right|\varphi_M(\gamma^G)\prod_{(j,k)\in G}\frac{e^{-it_{p,d}\gamma^G_{jk}}}{2\pi ip_0\gamma^G_{jk}}\frac{\partial\gamma^G_{jk}}{\partial X_{jk}}\mathrm{d}X. \right).$$
As we noted in Corollary \ref{density}, $$f_M(X):=\varphi_M(\gamma^G)\prod_{(j,k)\in G}\frac{e^{-it_{p,d}\gamma^G_{jk}}}{2\pi ip_0\gamma^G_{jk}}\frac{\partial\gamma^G_{jk}}{\partial X_{jk}}$$
is a product probability density, with each marginal density bounded above by a constant times the density of a centered normal with variance $\frac{1}{dt_{p,d}^2}$. So if we let $T$ have density $f_M$ and set $\Theta:=\gamma^G(T)$, we are interested in 
showing 
\begin{align}\label{mainerror}
\E\left[\left|\frac{\varphi_V(\Theta)}{\varphi_M(\Theta)}-1\right|\one\{\Theta\in\gamma^G(K)\}\right] \to 0.
\end{align}
To progress, we should start by bounding $\left|\frac{\varphi_V(\Theta)}{\varphi_M(\Theta)}-1\right|$ for fixed $\Theta\in\gamma^G(K)$. 
If we have $\|\Theta\|_{\op}< d$, then we can apply Lemma \ref{expectationdef}. Letting $U\in\D$ have density $\eta_d$, 
%
%
%
%
we can write 
\begin{align}\label{rat}
\varphi_V(\Theta)/\varphi_M(\Theta)-1&=i^{-n}\E\left[\big(R(U,\Theta)-1\big)\det\beta'(U)|\Theta\right], 
\end{align}
where 
$$R(U,\Theta):=\det\left(I-\frac{i}{d}\Theta\beta(U)^{-1}\right)^{-\frac{d}{2}}\etr\left(\frac{1}{4d}\Theta^2\right)$$ for fixed $\Theta\in\gamma^G(K)$. Before we start bounding $(\ref{rat})$, we will find the following lemma helpful for controlling expectations of functions of $U$.
\begin{lemma}{\label{normalcomp}}
For all $x\in\R$, we have $\eta_d(x)\leq\sqrt{\frac{d}{4\pi}}\exp\left(-\frac{d}{4}x^2+O(d^{-1})\right)$. 
\end{lemma}
\begin{proof}
This follows by Stirling's approximation and writing $x\cot x+\log\sinc x\leq1-\frac{x^2}{2}$. 
\end{proof}
Apropos of this bound and the bound in Corollary \ref{density}, we have the following lemma and corollary. 
\begin{lemma}{\label{subgaussian}}
Let $X_1,...,X_N$ be i.i.d. centered real random variables with density function $f$ satisfying 
$$f(x)\leq\frac{Ce^{-\frac{x^2}{2{s}^2}}}{\sqrt{2\pi{s}^2}}$$
for a fixed constant $C$. Set $X_*=\max_{j\leq N}|X_j|$. Then for $t<{s}^{-1}$ we have 
$$\E\exp\left(t^2X_*^2/2\right)\leq\frac{\exp\left(t^2{s}^2\log(CN)\right)}{1-t^2{s}^2}.$$
\end{lemma}
\begin{proof}
Set $g(x):=\exp\left(\frac{t^2x^2}{2}\right)$. Integrating by parts, applying a union bound, and then another integration by parts gives us \begin{align*}
\E g(X_*)&=g(0)+\int_0^\infty g'(x)\Pr[X_*\geq x]dx \\
&\leq g(0)+\int_0^\infty g'(x)\max\left(1,CN\erfc\left(\frac{x}{\sqrt2{s}}\right)\right)dx \\ 
&=\sqrt{\frac{2}{\pi{s}^2}}CN\int_{\sqrt{2}{s}\erfc^{-1}\left(\frac{1}{CN}\right)}^\infty \exp\left(-({s}^{-2}-t^2)x^2/2)\right)dx \\ 
&=\frac{CN}{\sqrt{1-t^2{s}^2}}\erfc\left(\sqrt{1-t^2{s}^2}\erfc^{-1}\left(\frac{1}{CN}\right)\right).
\end{align*}
It can be seen by expanding the function 
$\frac{\erfc(\sqrt{1-t^2{s}^2}x)}{\sqrt{1-t^2{s}^2}\erfc(x)}$
around $x=\infty$ that 
$$\frac{\erfc(\sqrt{1-t^2{s}^2}x)}{\sqrt{1-t^2{s}^2}\erfc(x)}\leq\frac{e^{t^2{s}^2x^2}}{1-t^2{s}^2}$$
for large $x$. Taking $x=\erfc^{-1}\left(\frac{1}{CN}\right)\leq\sqrt{\log(CN)}$ gives the desired result. 
\end{proof}
\begin{corollary}\label{subgaussianmoments}
Under the same assumptions as the previous lemma, For each $q\geq1$ there are constants $C_q$ independent of $N$, $C$, and ${s}$ such that
$\E X_*^q\leq C_q\left({s}^2\log(CN)\right)^{q/2}$. 
\end{corollary}

The following is the main remaining technical component of the proof.
\begin{proposition}\label{Abound}
If $G$ satisfies 
\[
  {(n\log n+\sigma_G)}\delta_G^2 = o(d^{-1}t_{p,d}^{-4}),
  \quad
  \sqrt{{\sigma_G}}\tau_G = o(d^{-1}t_{p,d}^{-3}),
  \quad\text{and}
  \quad \log n = O(dt_{p,d}^2),
\]
  then 
\[
\E\left[\left|\frac{\varphi_V(\Theta)}{\varphi_M(\Theta)}-1\right|\one\{\Theta\in\gamma^G(K)\}\right]
=o(1).
\]
\end{proposition}
\begin{proof}
Note first of all that for $\Theta\in\gamma^G(K)$, we have 
\[
  \|\Theta\|_{\op}
  \leq\|\Theta\|_{\Fr}
  = O(dt_{p,d}\sqrt{{\sigma_G}})
  = o\left(\frac{d}{\delta_G\sqrt{d}t_{p,d}}\right)
  = o(d).
\]
Thus we are justified in using (\ref{rat}) whenever $\Theta\in\gamma^G(K)$. 
Towards bounding the modulus of (\ref{rat}) in this case, we have 
\begin{align*}
\left|\varphi_V(\Theta)/\varphi_M(\Theta)-1\right|&\leq\sqrt{\E|\det\beta'(U)|^2\E[|R(U,\Theta)-1|^2|\Theta]} \\
&=\sqrt{\E|\det\beta'(U)|^2}\sqrt{\E\big[\big|R(U,\Theta)|-1\big|^2\rvert\Theta\big]+2\E\left[|R(U,\Theta)|-\Re R(U,\Theta)\rvert\Theta\right]}.
\end{align*}
We can do away with this first factor as follows. Note that $\eta_d(U)=\left(2\pi\left(1-\frac{2}{d}\right)^{\frac{d}{2}-1}\right)^n\eta_2(U)\eta_{d-2}(U)$ and $\eta_2(U)|\det\beta'(U)|^2\leq\eta_2(0)|\det\beta'(0)|^2=\left(\frac{e}{2\pi}\right)^n$. Thus 
$$\E|\det\beta'(U)|^2\leq\int_{\D_\pi}\left(e\left(1-\frac{2}{d}\right)^{\frac{d}{2}-1}\right)^n\eta_{d-2}(U)\mathrm{d}U\leq\int_{\D_\pi}e^{\frac{2n}{d}}\eta_{d-2}(U)\mathrm{d}U=e^{\frac{2n}{d}}.$$
Having done this, bounding (\ref{rat}) reduces to bounding  \begin{align*}
&e^{\frac{n}{d}}\sqrt{\E\big[\big|R(U,\Theta)|-1\big|^2\rvert\Theta\big]+2\E\left[|R(U,\Theta)|-\Re R(U,\Theta)\rvert\Theta\right]}
\end{align*}
for $\Theta\in\gamma^G(K)$.
Then by Jensen's inequality, bounding (\ref{mainerror}) is reduced to bounding
\begin{align}\label{inrat}
&e^{\frac{n}{d}}\sqrt{\E\left[\left(\big|R(U,\Theta)|-1\big|^2+2\big(|R(U,\Theta)|-\Re R(U,\Theta)\big)\right)\one\{\Theta\in\gamma^G(K)\}\right]}.
\end{align}
To do this, we first need an estimate of $L=L(U,\Theta):=\log R(U,\Theta)$. Let $L_r$ and $L_i$ be the real and imaginary parts of $L$ respectively. Then we have 
\begin{align*}
&\E\left[\left(\big|R(U,\Theta)|-1\big|^2+2\big(|R(U,\Theta)|-\Re R(U,\Theta)\big)\right)\one\{\Theta\in\gamma^G(K)\}\right] \\
=&\E\left[\left(e^{2L_r}-2e^{L_r}+1+2e^{L_r}(1-\cos L_i)\right)\one\{\Theta\in\gamma^G(K)\}\right] \\ 
\leq&\E\left[\left(e^{2L_r}-2L_r-1+e^{L_r}L_i^2\right)\one\{\Theta\in\gamma^G(K)\}\right] \\ 
\leq&\E\left[\left(e^{2L_r}-2L_r-1\right)\one\{\Theta\in\gamma^G(K)\}\right]+\sqrt{\E\left[e^{2L_r}\one\{\Theta\in\gamma^G(K)\}\right]\E\left[L_i^4\one\{\Theta\in\gamma^G(K)\}\right]}.
\end{align*}
We now finish by computing some expectations of $L$.  We have put these expectations into lemmas whose proofs we delay, and we show how these lemmas complete the proof.
By Lemmas \ref{bound1} and \ref{bound2}, 
$$\E\left[\left(e^{2L_r}-2L_r-1\right)\one\{\Theta\in\gamma^G(K)\}\right]=O\left(dt_{p,d}^4\sigma_G\delta_G^2+\sqrt{\sigma_G}dt_{p,d}^3\tau_G\right).$$
By Lemma \ref{bound3}, 
$$\sqrt{\E\left[e^{2L_r}\one\{\Theta\in\gamma^G(K)\}\right]\E\left[L_i^4\one\{\Theta\in\gamma^G(K)\}\right]}=O\left(d^2t_{p,d}^{6}\max\left(\frac{\delta_G^2n\log n}{dt_{p,d}^2},\,\sigma_G^2\delta_G^4t_{p,d}^2,\,\tau_G^2\right)\right).$$
Since our assumptions on $G$ imply that $dt_{p,d}^4\sigma_G^2\delta_G^2+\sqrt{\sigma_G}dt_{p,d}^3\tau_G=o(1)$, the only term from this maximum that could compete with $dt_{p,d}^4\sigma_G\delta_G^2+\sqrt{\sigma_G}dt_{p,d}^3\tau_G$ is $ndt_{p,d}^4\delta_G^2\log n$. So in total, 
$$\E\left[\left|\frac{\varphi_V(\Theta)}{\varphi_M(\Theta)}-1\right|\one\{\Theta\in\gamma^G(K)\}\right]=O\left(\sqrt{dt_{p,d}^3(t_{p,d}\delta_G^2(n\log n+ {\sigma_G})+\sqrt{{\sigma_G}}\tau_G)}\right)=o(1).$$
\end{proof}

\begin{lemma}\label{logbound}
Let $U_*:=\max_{j\in[n]}|U_j|$ and $\Theta_*:=\max_{(j,k)\in G}|\Theta_{jk}|$. Then, for $\Theta$ with $\|\Theta\|_{\op}< d$, we have 
$$L(U,\Theta)=\frac{i}{2d}\tr\left(\Theta^2U\right)+E(U,\Theta)$$ where $$|E(U,\Theta)| = O\left(\frac{\|\Theta\|_{\Fr}^2}{d}U_*^2+\frac{\delta_G^2\|\Theta\|_{\Fr}^2}{d^3}\Theta_*^2+\frac{\tau_G}{d^2}\Theta_*^3\right).$$
\end{lemma} 
\begin{proof}
Since
$\|\Theta\|_{\op}< d$, we have 
\begin{align*} 
\det\left(I-\frac{i}{d}\Theta\beta(U)^{-1}\right)^{-\frac{d}{2}} =\etr\left(-\frac{1}{4d}\big(\Theta\beta(U)^{-1}\big)^2-\frac{i}{6d^2}\big(\Theta\beta(U)^{-1}\big)^3\right)\exp\left(O\left(\frac{\|\Theta\|_{\op}^2\|\Theta\|_{\Fr}^2}{d^3}\right)\right). 
\end{align*}
Therefore 
$$\left|L-\frac{i}{2d}\tr\left(\Theta^2U\right)\right|\leq\frac{1}{4d}\left|\tr\left(\Theta^2(I-2iU)-(\Theta\beta(U)^{-1})^2\right)\right|+\frac{1}{6d^2}\left|\tr\big(\Theta\beta(U)^{-1}\big)^3\right|+O\left(\frac{\|\Theta\|_{\op}^2\|\Theta\|_{\Fr}^2}{d^3}\right).$$
To bound this first term, we write
$$\tr\left((\Theta\beta^{-1})^2\right)=\sum_{j,k}\Theta_{jk}^2\beta_{jj}^{-1}\beta_{kk}^{-1}=\sum_{j,k}\Theta_{jk}^2(1-iU_{jj}-iU_{kk}+O(U_*^2))=\tr\big(\Theta^2(I-2iU)\big)+O\left(\|\Theta\|_{\Fr}^2U_*^2\right).$$
For the cubic term, we have 
\begin{align*}
\frac{1}{6d^2}\left|\tr\big(\Theta\beta(U)^{-1}\big)^3\right|\leq\frac{1}{6d^2}\sum_{j,k,\ell}|\Theta_{jk}\Theta_{k\ell}\Theta_{\ell j}\beta^{-1}_{jj}\beta^{-1}_{kk}\beta^{-1}_{\ell\ell}|\leq\frac{\tau_G}{d^2}\Theta_*^3.
\end{align*}
So 
$$\left|L-\frac{i}{2d}\tr\left(\Theta^2U\right)\right|=O\left(\frac{\|\Theta\|_{\Fr}^2}{d}U_*^2+\frac{\tau_G}{d^2}\Theta_*^3+\frac{\delta_G^2\|\Theta\|_{\Fr}^2}{d^3}\Theta_*^2\right).$$
\end{proof}
\begin{lemma}\label{bound1}
Under the assumptions of Proposition \ref{Abound}, 
$$\E[|R|^2\one\{\Theta\in\gamma^G(K)\}]=
1+O\left(dt_{p,d}^4{\sigma_G}\delta_G^2+dt_{p,d}^3\sqrt{{\sigma_G}}\tau_G\right).$$
\end{lemma}
\begin{proof}
By the previous lemma, 
we can choose a constant $C>0$ such that
\begin{align*}
|R|^2&\leq\exp\left(-\frac{1}{d}\Image\tr(\Theta^2U)+C\left(\frac{\|\Theta\|_{\Fr}^2}{d}U_*^2+\frac{\tau_G}{d^2}\Theta_*^3+\frac{\delta_G^2\|\Theta\|_{\Fr}^2}{d^3}\Theta_*^2\right)\right) \\ 
&\leq\exp\left(-\frac{1}{d}\sum_j\left(\sum_k\Image\Theta_{jk}^2\right)U_{jj}+C\left(\frac{\|\Theta\|_{\Fr}^2}{d}U_*^2+\frac{\tau_G\|\Theta\|_{\Fr}\Theta_*^2}{d^2}+\frac{\delta_G^2\|\Theta\|_{\Fr}^2}{d^3}\Theta_*^2\right)\right).
\end{align*}
Assuming $\Theta\in\gamma^G(K)$, we know that $\|\Theta\|_{\Fr}\leq O(\sqrt{{\sigma_G}}dt_{p,d})$. So by Cauchy--Schwarz and the independence of the entries of $U$,
\begin{align*}
&\E\left[\exp\left(-\frac{1}{d}\sum_j\left(\sum_k\Image\Theta_{jk}^2\right)U_{jj}+C'\sigma_G dt_{p,d}^2U_*^2\right)\Bigg\vert\Theta\right] \\
\leq&\sqrt{\E\exp\left(2C'\sigma_G dt_{p,d}^2U_*^2\right)\prod_{j}\E\left[\exp\left(-\frac{2}{d}\left(\sum_k\Image\Theta_{jk}^2\right)U_{jj}\right)\bigg\vert\Theta\right]} 
\end{align*}
for some other constant $C'$ and $\Theta\in\gamma^G(K)$. This first expectation can be bounded via Lemma \ref{subgaussian} by 
$$
\frac{\exp\left(O\left(\sigma_G t_{p,d}^2\log n)\right)\right)}
{1-O(\sigma_G t_{p,d}^2)}
=
\exp\left(O\left(\sigma_G t_{p,d}^2\log n)\right)\right)=
\exp\left(O\left(\sigma_G dt_{p,d}^4\right)\right)
$$
since ${\sigma_G}t_{p,d}^2< {\sigma_G}dt_{p,d}^4\delta_G^2+\sqrt{{\sigma_G}}dt_{p,d}^3\tau_G=o(1)$ and $\log n = O(dt_{p,d}^2)$.
Recall that we are writing $\Theta=\gamma^G(T)$ where $T$ is a random matrix with density
$$f_M(X):=\varphi_M(\gamma^G)\prod_{(j,k)\in G}\frac{e^{-it_{p,d}\gamma^G_{jk}}}{2\pi ip_0\gamma^G_{jk}}\frac{\partial\gamma^G_{jk}}{\partial X_{jk}}\leq\prod_{(j,k)\in G}\sqrt{\frac{C^2dt_{p,d}^2}{2\pi}}e^{-\frac{dt_{p,d}^2}{2}X_{jk}^2}$$
for an absolute constant $C$. So writing $\Theta=\gamma^G(T)$ with $T\in K$, 
we have 
$$\sum_k\Image\Theta_{jk}^2\leq 3d^2t_{p,d}^2\sum_k|T_{jk}|\leq3d^2t_{p,d}^2\sqrt{\sum_k|T_{jk}|^2}.$$
So by Lemma \ref{normalcomp} we have 
\begin{align*}
\prod_j\E\left[\exp\left(-\frac{2}{d}\left(\sum_k\Image\Theta_{jk}^2\right)U_{jj}\right)\Bigg\vert\Theta\right]&=\exp\left(O\left(\frac{n}{d}\right)+\frac{4}{d^3}\sum_j\left(\sum_k\Image\Theta_{jk}^2\right)^2\right) \\
&\leq\exp\left(O\left(\frac{n}{d}\right)+36dt_{p,d}^4\|T\|_{\Fr}^2\right)=\exp\left(O\left(dt_{p,d}^4\sigma_G\right)\right).
\end{align*}
We have succeeded in showing that
$$\E[|R|^2|\Theta]=\exp\left(O\left(dt_{p,d}^4\sigma_G+\frac{t_{p,d}^2\sigma_G\delta_G^2+t_{p,d}\sqrt{\sigma_G}\tau_G}{d}\Theta_*^2\right)\right)$$
for $\Theta\in\gamma^G(K)$ from which it follows that 
$$\E[|R|^2|T]=\exp\left(O\left(dt_{p,d}^4\sigma_G\right)+C(dt_{p,d}^4\sigma_G\delta_G^2+dt_{p,d}^3\sqrt{\sigma_G}\tau_G)(1+o(1)+T_*^2)\right)$$
for $T\in K$ and a fixed constant $C$.
Since $\E T^2\log \sigma_G = O(\frac{\log n}{dt_{p,d}^2}) = O(1)$, by Lemma \ref{subgaussian} we have
$$\E[|R|^2\one\{T\in K\}]=\frac{\exp\left(O\left(dt_{p,d}^4\sigma_G\delta_G^2+dt_{p,d}^3\sqrt{\sigma_G}\tau_G\right)\right)}{1-O\left(\sqrt{\sigma_G}t_{p,d}\left(t_{p,d}\sqrt{\sigma_G}\delta_G^2+\tau_G\right)\right)}$$
Since our assumptions on $G$ imply that $dt_{p,d}^4\sigma_G\delta_G^2+dt_{p,d}^3\sqrt{\sigma_G}\tau_G = o(1)$, the result now follows.
\end{proof}
\begin{lemma}\label{bound2}
Under the assumptions of Proposition \ref{Abound}, 
$$-\E[L_r(U,\Theta)\one\{\Theta\in\gamma^G(K)\}]=O\left(dt_{p,d}^4\sigma_G\delta_G^2+\tau_Gdt_{p,d}^3\right).$$
\end{lemma}
\begin{proof}
By Lemma \ref{logbound},
$$-L_r(U,\Theta)\one\{\Theta\in\gamma^G(K)\}\leq\frac{1}{2d}\Image\tr(\Theta^2U)\one\{\Theta\in\gamma^G(K)\}+C\left(dt_{p,d}^2\sigma_G U_*^2+\frac{t_{p,d}^2\sigma_G\delta_G^2}{d}\Theta_*^2+\frac{\tau_G}{d^2}\Theta_*^3\right)$$
for a large enough constant $C$. So by Corollary \ref{subgaussianmoments}
\begin{align*}
-\E[L_r(U,\Theta)\one\{\Theta\in\gamma^G(K)\}]
&\leq C\left({\sigma_G}dt_{p,d}^2\E U_*^2+\frac{{\sigma_G}t_{p,d}^2\delta_G^2}{d}\E\Theta_*^2+\frac{\tau_G}{d^2}\E\Theta_*^3\right) \\ 
&\leq C\left({\sigma_G}t_{p,d}^2\log n+{\sigma_G}dt_{p,d}^4\delta_G^2(1+o(1)+\E T_*^2)+\tau_Gdt_{p,d}^3\E(1+o(1)+T_*^2)^{3/2}\right) \\ 
&=O\left({\sigma_G}dt_{p,d}^4\delta_G^2(1+\E T_*^2)+\tau_Gdt_{p,d}^3\E(1+T_*^2)^{3/2}\right).
\end{align*}
By Corollary \ref{subgaussianmoments}, 
$\E T_*^2 = O(\frac{\log n}{dt_{p,d}^2}) = O(1)$ and
$$\E(1+T_*^2)^{3/2}\leq\sqrt{\E(1+T_*^2)^3}= O\left((1+\E T_*^2)^{3/2}\right)=O(1).$$
\end{proof}
\begin{lemma}\label{bound3}
Under the assumptions of Proposition \ref{Abound}, we have 
$$\E\left[ L_i(U,\Theta)^4\one\{\Theta\in\gamma^G(K)\}\right]=O\left(d^4t_{p,d}^{12}\max\left(\frac{\delta_G^2n\log n}{dt_{p,d}^2},\,\sigma_G^2\delta_G^4t_{p,d}^2,\,\tau_G^2\right)^2\right).$$
\end{lemma}
\begin{proof}
By Lemma \ref{logbound} we have \begin{align*}
L_i(U,\Theta)&=-\frac{1}{2d}\tr(\Re\Theta^2U)+E(U,\Theta)
\end{align*}
where
$$|E(U,\Theta)|=O\left(\frac{\|\Theta\|_{\Fr}^2}{d}U_*^2+\frac{\delta_G^2\|\Theta\|_{\Fr}^2}{d^3}\Theta_*^2+\frac{\tau_G}{d^2}\Theta_*^3\right).$$
In particular, $$|E(U,\Theta)|\one\{\Theta\in\gamma^G(K)\}=O\left({\sigma_G}dt_{p,d}^2U_*^2+\frac{{\sigma_G}\delta_G^2t_{p,d}^2}{d}\Theta_*^2+\frac{\tau_G}{d^2}\Theta_*^3\right).$$
Let $\F$ be the sigma--algebra generated by $\Theta$ and the value of $U_*$. To condense notation, set 
$$A_j:=-\frac{1}{2d}\left(\sum_k\Re\Theta_{jk}^2\right).$$
and 
$$C:=\E[E(U,\Theta)|\F]\one\{\Theta\in\gamma^G(K)\}.$$
so that $A_j$ and $C$ are both $\F$--measurable. Then the entries of $U$ conditioned on $\F$ remain symmetric and independent. Thus, for $\Theta\in\gamma^G(K)$, we have 
\begin{align*}
\E[L_i(U,\Theta)^4|\F]&=\E\bigg[\bigg(\sum_jA_jU_{jj}\bigg)^4\bigg\vert\F\bigg]+6C^2\E\bigg[\bigg(\sum_jA_jU_{jj}\bigg)^2\bigg\vert\F\bigg]+C^4 \\
&=O\left(n^2U_*^4\max_j A_j^4+C^2nU_*^2\max_j A_j^2+C^4\right) \\
&=O\left(n^2\delta_G^4d^{-4}U_*^4\Theta_*^8+C^2n\delta_G^2d^{-2}U_*^2\Theta_*^4+C^4\right).
\end{align*}
By Corollary \ref{density}, Lemma \ref{normalcomp}, and Corollary \ref{subgaussianmoments}, 
for any polynomial $P$ in two variables with nonnegative coefficients, we have  
$$\E P(U_*,\Theta_*)=O\big(P(\E U_*,\E\Theta_*)\big).$$
Since we can bound $C$ by such a polynomial, we find that
$$\E\left[ L_i(U,\Theta)^4\one\{\Theta\in\gamma^G(K)\}\right]=O\left(n^2\delta_G^4d^{-4}(\E U_*)^4(\E\Theta_*)^8+(\E C)^2n\delta_G^2d^{-2}(\E U_*)^2(\E\Theta_*)^4+(\E C)^4\right).$$
We have $\E U_*=O(\sqrt{\log n/d})$, $\E\Theta_*=O\left(dt_{p,d}\right)$, and 
\begin{align*}
\E C&=O\left(\sigma_G dt_{p,d}^2\E U_*^2+\frac{\sigma_G \delta_G^2t_{p,d}^2}{d}\E\Theta_*^2+\frac{\tau_G}{d^2}\E\Theta_*^3\right) \\ 
&=O\left(\sigma_G t_{p,d}^2\log n+\sigma_G\delta_G^2dt_{p,d}^4+\tau_Gdt_{p,d}^3\right) \\ 
&=O\left(dt_{p,d}^3(\sigma_G\delta_G^2t_{p,d}\vee\tau_G)\right). 
\end{align*}
So 
\begin{align*}
\E\left[ L_i(U,\Theta)^4\one\{\Theta\in\gamma^G(K)\}\right]&=O\left(n^2\delta_G^4d^{-6}\log^2n(\E\Theta_*)^8+(\E C)^2n\delta_G^2d^{-3}\log n(\E\Theta_*)^4+(\E C)^4\right) \\ 
&=O\left(n^2\delta_G^4d^2t_{p,d}^8\log^2n+(\E C)^2n\delta_G^2dt_{p,d}^4\log n+(\E C)^4\right) \\ 
&=O\left(n^2\delta_G^4d^2t_{p,d}^8\log^2n\left(1+\frac{(\E C)^2}{n\log n\delta_G^2dt_{p,d}^4}+\left(\frac{(\E C)^2}{n\log n\delta_G^2dt_{p,d}^4}\right)^2\right)\right) \\ 
&=O\left(n^2\delta_G^4d^2t_{p,d}^8\log^2n\left(1\vee\frac{(\E C)^2}{n\log n\delta_G^2dt_{p,d}^4}\right)^2\right) \\ 
&=O\left(n^2\delta_G^4d^2t_{p,d}^8\log^2n\left(1\vee\frac{dt_{p,d}^2(\sigma_G\delta_G^2t_{p,d}\vee\tau_G)^2}{n\log n\delta_G^2}\right)^2\right) \\ 
&=O\left(d^4t_{p,d}^{12}\max\left(\frac{\delta_G^2n\log n}{dt_{p,d}^2},\,\sigma_G^2\delta_G^4t_{p,d}^2,\,\tau_G^2\right)^2\right).
\end{align*}
\end{proof}

\begin{proof}[Proof of Theorem \ref{mainmain}]
The only thing that needs to be checked here is that the conditions of Theorem \ref{mainmain} imply all of the conditions for Lemma \ref{Bbound} and Proposition \ref{Abound}.
Recall the conditions assumed in Theorem \ref{mainmain} were that $\log d=O(\log p^{-1})$, $n,d\to\infty$, and that 
$G$ satisfies 
\[
  (\sigma_G+\mu_G\log \mu_G)\delta_G^2 = o\left(\frac{d}{\log^2p^{-1}}\right),
  \quad
  \sigma_G\tau_G^2= o\left(\frac{d}{\log^3p^{-1}}\right),
  \quad\text{and}
  \quad\frac{\mu_G\delta_G^4}{\sigma_G}= O\left(\frac{d\log d}{\log^2p^{-1}}\right).
\]
From Lemma \ref{lem:pp0}, we can compare $\log p^{-1}= \Theta(d t_{p,d})$, and thus 
we have
\begin{multicols}{2}
\begin{enumerate}
    \item 
    $\log d= O(dt_{p,d}^2)$, 
    \item
    $(\sigma_G+\mu_G\log \mu_G)\delta_G^2=o\left(d^{-1}t_{p,d}^{-4}\right)$,
    \item 
    $\sigma_G\tau_G^2= o\left(d^{-2}t_{p,d}^{-6}\right)$, 
    \item 
    $\frac{\mu_G \delta_G^4}{\sigma_G}= O\left(d^{-1}t_{p,d}^{-4}\log d\right)$.
\end{enumerate}
\end{multicols}

We will start with the conditions of Proposition \ref{Abound}. 
The conditions 
$\sigma_G\delta_G^2=o(d^{-1}t_{p,d}^{-4})$ and   $\sqrt{\sigma_G}\tau_G=o(d^{-1}t_{p,d}^{-3})$ follow from conditions 2 and 3 above. 
Conditions 1 and 2 give us $\log \mu_G=O(\log d)=O(dt_{p,d}^2)$. 

As for Lemma \ref{Bbound}, we need to show that conditions 1 through 4 imply that 
\[
  \log d= O(dt_{p,d}^2),
  \quad \sigma_G= o(t_{p,d}^{-2}),
  \quad \delta_G= o(t_{p,d}^{-1}),
  \quad \frac{\mu_G \delta_G^4}{\sigma_G}= O\left(\frac{\log d}{dt_{p,d}^4}\right),
  \quad\text{and}\quad 
  \frac{\tau_G}{\sigma_G}=  O\left(\frac{\log d}{dt_{p,d}^3}\right).
\]
Both $\log d= O(dt_{p,d}^2)$ and 
$\frac{\mu_G \delta_G^4}{\sigma_G} = o\left(\frac{\log d}{dt_{p,d}^4}\right)$ are just conditions 1 and 4 respectively. Condition 2 implies both that 
$\sigma_G = o\left(d^{-1}t_{p,d}^{-2}\delta_G^{-2}t_{p,d}^{-2}\right)= o(t_{p,d}^{-2})$ and that $\delta_G= o(t_{p,d}^{-1})$. Finally, condition 3 tells us that  
$$\frac{\tau_G}{\sigma_G}= o\left(\frac{\sigma_G^{-3/2}}{dt_{p,d}^3}\right)=O\left(\frac{\log d}{dt_{p,d}^3}\right),$$
which completes the proof. 
\end{proof}

\appendix

\section{Path Integration and Deformation}
In this section we review the necessary results that allow us to use these techniques. For this section, $U$ will always denote a nonempty open subset of $\C$ unless it is stated otherwise and $-\infty<a<b<\infty$.  
\begin{definition}
A map $f:U\to\C$ is said to be holomorphic if and only if the limit $\lim_{z\to z_0}\frac{f(z)-f(z_0)}{z-z_0}$ exists and is finite for all $z_0\in U$. In the special case that $U=\C$, $f$ is called entire. 
\end{definition} 
\begin{definition}
Given a closed interval $[a,b]$ and points $A,B\in\C$, a path from $A$ to $B$ is a continuous map $\alpha:[a,b]\to\C$ that's real and imaginary parts are both piecewise differentiable over $(a,b)$ and which satisfies $\alpha(a)=A$ and $\alpha(b)=B$. In the special case that $A=B$, $\alpha$ is called a contour. 
\end{definition}
\begin{definition}
Two paths $\alpha,\beta:[a,b]\to U$ are called smoothly homotopic in $U$ if there is a continuous map $H:[a,b]\times[0,1]\to U$ such that $H(x,0)=\alpha(x)$, $H(x,1)=\beta(x)$, and which is continuously differentiable in the second variable whenever the first variable is in $\{a,b\}$. This map $H$ is called a smooth homotopy from $\alpha$ to $\beta$. In this context, $\alpha$ is called the initial path and $\beta$ is called the target path. If $f:U\to\C$ is holomorphic, we may say that $\alpha$ and $\beta$ are homotopic with respect to $f$. 
\end{definition}
\begin{definition}
For a path $\alpha:[a,b]\to U$ and a holomorphic function $f:U\to\C$, the path integral of $f$ over $\alpha$ is defined as 
$$\int_\alpha f(z)dz:=\int_a^bf(\alpha(x))\alpha'(x)dx.$$
\end{definition}
\begin{definition}
For $f:U\to\C$ Lebesgue--measurable and $\alpha:[a,b]\to U$ a path, we will write $f\in\mathcal{L}^1(\alpha)$ to indicate that 
$$\int_\alpha|f(z)|dz:=\int_a^b|f(\alpha(x))\alpha'(x)|dx<\infty.$$
In the special case where $-a=b$ is allowed to tend to $\infty$, we'll write $f\in\mathcal{L}^1(\alpha)$ if and only if 
$$\lim_{b\to\infty}\int_{-b}^b|f(\alpha(x))\alpha'(x)|dx<\infty.$$ 
\end{definition}
\begin{theorem}[Path Independence] 
Let $f:U\to\C$ be holomorphic and $\alpha,\beta:[a,b]\to U$ be two paths from $A$ to $B$ which are smoothly homotopic in $U$. Then 
$$\int_\alpha f(z)dz=\int_\beta f(z)dz.$$
\end{theorem}
A slightly stronger result than this is stated and proved in Theorem 6.13 of \cite{Complex}. 
\begin{corollary}{\label{moving}}
Let $f:U\to\C$ be holomorphic and $\alpha,\beta:[a,b]\to U$ be two paths which do not necessarily have the same endpoints. Suppose $H:[a,b]\times[0,1]\to U$ is a smooth homotopy between $\alpha$ and $\beta$ and define paths $\gamma,\delta:[0,1]\to U$ by $\gamma(y):=H\left(a,y\right)$ and $\delta(y):=H\left(b,y\right)$. 
Then 
$$\int_\alpha f(z)dz-\int_\beta f(z)dz=\int_\gamma f(z)dz-\int_\delta f(z)dz.$$
\end{corollary}
\begin{proof}
Note that $\alpha(b)=\delta(0)$ and $\beta(a)=\gamma(1)$. As such, we can write 
$$(\alpha+\delta)(t):=
\begin{cases}
\alpha\left((b-a)t+b\right),& t\in[-1,0] \\ 
\delta(t),& t\in[0,1]  
\end{cases}$$ 
to denote the concatenation of $\alpha$ and $\delta$, and we can define $\beta+\gamma$ analogously. Thus we have two homotopic paths $\alpha+\delta$ and $\beta+\gamma$ between the points $\alpha(a)=\gamma(0)$ and $\beta(b)=\delta(1)$. So path independence gives us that 
$$\int_{\alpha+\delta}f(z)dz=\int_{\beta+\gamma}f(z)dz.$$
We leave this part to the reader to check for themselves, but it follows from the definition of a path integral and standard manipulations of Riemann integrals that 
$$\int_\alpha f(z)dz+\int_\delta f(z)dz=\int_{\alpha+\delta}f(z)dz=\int_{\beta+\gamma}f(z)dz=\int_\beta f(z)dz+\int_\gamma f(z)dz$$
from which the result follows. 
\end{proof}
\begin{corollary}
Let $f:U\to\C$ be holomorphic and $\alpha,\beta:[-T,T]\to U$ be two paths which are homotopic over $U$. Suppose $H:[-T,T]\times[0,1]\to U$ is a smooth homotopy between $\alpha$ and $\beta$ such that $\left|\frac{\partial}{\partial y}H(\pm T,y)\right|\leq M$. Then 
$$\left|\int_\alpha f(z)dz-\int_\beta f(z)dz\right|\leq M\sup_{y\in[0,1]}\big|f\big(H(-T,y)\big)-f\big(H(T,y)\big)\big|.$$ 
\end{corollary}

Finally, we would like a way to cleanly extend these deformation results to functions of more than one complex variable. Let $U$ be an open subset of $\C^N$ and suppose that $f:U\to\C$ is holomorphic in each variable. By this we mean that if we fix the coordinates $z_k$ for $k\neq j$, then the single variable map 
$$f_j(z):=f(z_1,...,z_{j-1},z,z_{j+1},...,z_N)$$
is holomorphic over $
\{z\in\C:(z_1,...,z_{j-1},z,z_{j+1},...,z_N)\in U\}$ whenever this set is nonempty. We can define paths in $\C^N$ (at least the ones we will be interested in) as $N$--ary direct products of paths in $\C$. We can then define via Fubini's theorem the path integral of $f:U\to\C$ over $\prod_{j=1}^N\alpha_j:\prod_{j=1}^N[a_j,b_j]\to U$ in this setting by writing 
$$\int_{\prod_{j=1}^N\alpha_j}f(z)dz:=\int_{a_1}^{b_1}\cdots\int_{a_N}^{b_N}f\big(\alpha_1(x_1),...,\alpha_N(x_N)\big)\prod_{j=1}^N\alpha_j'(x_j)dx.$$ 

The next lemma gives us sufficient conditions for being able to deform such paths in the special circumstances that are relevant for our purposes. For a path $\alpha:[a,b]\to\C$, define $\alpha^k:[a,b]^k\to\C^k$ by $\alpha^k(x_1,...,x_k):=\big(\alpha(x_1),...,\alpha(x_k)\big)$. For $z\in\C^N$ and $1\leq k\leq N$, let $\hat z_k:=(z_1,...,z_{k-1},z_{k+1},...,z_N)\in\C^{N-1}$. 
\begin{lemma}\label{worldmover}
Let $U\subseteq\C^N$ be open and let $f:U\to\C$ be holomorphic. Let $\alpha,\beta:\R\to\C$ be piecewise differentiable and let $\alpha_T,\beta_T$ be the respective restrictions of these functions to $[-T,T]$ which, for each $T>0$, are homotopic over $\C$ via the smooth homotopy $H$ with target $\beta$ and $\big|\frac{\partial}{\partial y}H(\pm T,y)\big|\leq M$ for all $T>0$. Suppose that $\big(H(x_1,y_1),...,H(x_N,y_N)\big)\in U$ for all $(x_1,...,x_N,y_1,...,y_N)\in\R^N\times[0,1]^N$. Suppose for $0\leq k\leq N$ that $f\in\mathcal{L}^1(\alpha^{N-k}\times\beta^k)$
and
$$\lim_{T\to\infty}\sup_{y\in[0,1]}|f_k(H(\pm T,y))|=0$$
whenever $\hat z_k\in\alpha^{N-k}\times\beta^{k-1}$. Then 
$$\lim_{T\to\infty}\left|\int_{\alpha_T^N}f(z)dz-\int_{\beta_T^N}f(z)dz\right|=0.$$
\end{lemma}
\begin{proof}
We have 
$$\left|\int_{\alpha_T^N}f(z)dz-\int_{\beta_T^N}f(z)dz\right|\leq\sum_{k=1}^N\left|\int_{\alpha_T^{N-k+1}\times\beta_T^{k-1}}f(z)dz-\int_{\alpha_T^{N-k}\times\beta_T^k}f(z)dz\right|.$$ 
We will focus on just the $k$th summand. By Fubini--Tonelli, we can rearrange our integrals to find that 
\begin{align*}
\left|\int_{\alpha_T^{N-k+1}\times\beta_T^{k-1}}f(z)dz-\int_{\alpha_T^{N-k}\times\beta_T^k}f(z)dz\right|\leq\int_{\alpha^{N-k}\times\beta^{k-1}}\left|\int_{\alpha_T}f(z)dz_k-\int_{\beta_T}f(z)dz_k\right|d\hat z_k.  
\end{align*}
By dominated convergence, our assumptions about $f$ and $f_k$, and the previous corollary, this tends to 0 as $T\to\infty$. 
\end{proof}

\newpage 

\printbibliography[heading=bibliography]

\end{document}